\documentclass{amsart}

\usepackage{geometry, comment}
\usepackage{parskip}
\usepackage{graphicx}
\usepackage{amssymb,amsmath,amsthm}
\usepackage{hyperref}
\usepackage{pdfsync}
\usepackage{footmisc}
\usepackage[all]{xy}

\newcommand{\BN}{\mathbb{N}}

\newcommand{\BZ}{\mathbb{Z}}
\newcommand{\GG}{\mathbb{G}}
\newcommand{\BA}{\ensuremath{\mathbb{A}}}
\newcommand{\ov}{\overline}

\newtheorem{theorem}{Theorem}[section]
\newtheorem{lemma}[theorem]{Lemma}
\newtheorem*{theorem3}{Theorem \ref{thm:minimal}}
\newtheorem*{theorem4}{Theorem \ref{thm:crit}}
\newtheorem{prop}[theorem]{Proposition}

\newtheorem{corollary}[theorem]{Corollary}
\newtheorem*{cor*}{Corollary}
\newtheorem*{claim}{Claim}
\DeclareMathOperator{\az}{alph}

\newtheorem{defn}[theorem]{Definition}
\newtheorem{expl}[theorem]{Example}

\title[Commensurability of RAAGs]{On commensurability of right-angled Artin groups I: RAAGs defined by trees of diameter $4$}

\author[M. Casals-Ruiz]{Montserrat Casals-Ruiz}
\address{Ikerbasque - Basque Foundation for Science and Matematika Saila,  UPV/EHU,  Sarriena s/n, 48940, Leioa - Bizkaia, Spain}
\email{montsecasals@gmail.com}

\author[I. Kazachkov]{Ilya Kazachkov}
\address{Ikerbasque - Basque Foundation for Science and Matematika Saila,  UPV/EHU,  Sarriena s/n, 48940, Leioa - Bizkaia, Spain}
\email{ilya.kazachkov@gmail.com}

\author[A. Zakharov]{Alexander Zakharov}
\address{Matematika Saila, UPV/EHU, Sarriena s/n, 48940, Leioa - Bizkaia, Spain, and Russian foreign trade academy, 4a Pudovkin Street, 119285, Moscow, Russia}
\email{zakhar.sasha@gmail.com}


\begin{document}

\thanks{This work was supported by the ERC Grant 336983, by the Basque Government grant IT974-16 and by the grant MTM2014-53810-C2-2-P of the Ministerio de Economia y Competitividad of Spain. The first author was supported by Juan de la Cierva programme of the Spanish Government. The third author was supported by the Russian Foundation for Basic Research (project no.  15-01-05823), and by CMUP (UID/MAT/00144/2013), which is funded by FCT (Portugal) with national (MEC) and European structural funds (FEDER), under the partnership agreement PT2020.}

\begin{abstract}
In this paper we study the classification of right-angled Artin groups up to commensurability. We characterise the commensurability classes of RAAGs defined by trees of diameter  4. In particular, we prove a conjecture of Behrstock and Neumann that there are infinitely many commensurability classes. Hence, we give first examples of RAAGs that are quasi-isometric but not commensurable.
\end{abstract}
\maketitle

\section{Introduction}
	\subsection{Context}
	One of the basic problems on locally compact topological groups is to classify their lattices up to commensurability. Recall that two lattices $\Gamma_1, \Gamma_2 <G$ are commensurable if and only if there exists $g\in G$ such that $\Gamma_1 \cap \Gamma_2^g$ has finite index in both $\Gamma_1$ and $\Gamma_2^g$. In particular, commensurable lattices have covolumes that are commensurable real numbers, that is, they have a rational ratio.
	
	The notion of commensurability was generalized to better suit topological and large-scale geometric properties and to compare groups without requiring them to be subgroups of a common group. More precisely, we say that two groups $H$ and $K$ are (abstractly) commensurable if they have isomorphic finite index subgroups. In this article, we will only be concerned with the notion of abstract commensurability and we simply refer to it as commensurability.
	
	As we mentioned, commensurability is closely related to the large-scale geometry of the group. Indeed, any finitely generated group can be endowed with a natural word-metric which is well-defined up to quasi-isometry and since any finitely generated group is quasi-isometric to any of its finite index subgroups, it follows that commensurable groups are quasi-isometric.
	
	Gromov suggested to study groups from this geometric point of view and understand the relation between these two concepts. More precisely, a basic problem in geometric group theory is to classify commensurability and quasi-isometry classes (perhaps within a certain class) of finitely generated groups and to understand whether or not these classes coincide.
	
	The classification of groups up to commensurability (both in the abstract and classical case) has a long history and a number of famous solutions for very diverse classes of groups such as Lie groups, hyperbolic 3-manifold groups, pro-finite groups, Grigorchuk-Gupta-Sidki groups, etc, see for instance \cite{5, 13, 20, 25, 26, GrW, Gar}.

	In this paper, we focus on the question of classification of right-angled Artin groups, RAAGs for short, up to commensurability. Recall that a RAAG is a finitely presented group $\GG(\Gamma)$ which can be described by a finite simplicial graph $\Gamma$, the commutation graph, in the following way: the vertices of $\Gamma$ are in bijective correspondence with the generators of $\GG(\Gamma)$ and the set of defining relations of $\GG(\Gamma)$ consists of commutation relations, one for each pair of generators connected by an edge in $\Gamma$.
	
	RAAGs have become central in group theory, their study interweaves geometric group theory with other areas of mathematics. This class interpolates between two of the most classical families of groups, free and free abelian groups, and its study provides uniform approaches and proofs, as well as rich generalisations of the results for free and free abelian groups. The study of this class from different perspectives has contributed to the development of new, rich theories such as the theory of CAT(0) cube complexes and has been an essential ingredient in Agol's solution to the Virtually Fibered Conjecture.
	
	The commensurability classification of RAAGs has been previously solved for the following classes of RAAGs:
	\begin{itemize}
		\item	Free groups \cite{56, 25, 47}, \cite[1.C]{31};
		\item	Free Abelian groups, \cite{30, 3};
		\item   $F_m \times \mathbb{Z}^n$, \cite{58};
		\item	Free products of free groups and free Abelian groups, \cite{5};
		\item	$F_m \times F_n$ with $m,n \ge 2$, \cite{60, 17};
		\item	$\GG(\Gamma)$, where $\Gamma$ is a tree of diameter $\le 3$, \cite{BN};
		\item   $\GG(\Gamma)$, where $\Gamma$ is connected, triangle- and square-free graph without any degree one vertices, \cite{KKi}
		\item	$\GG(\Gamma)$, when the outer automorphism of $\GG$ is finite, $\Gamma$ is star-rigid and does not have induced 4-cycles, \cite{Huang}.
	\end{itemize}
	
	It turns out that, inside the class of RAAGs, the classification up to commensurability coincides with the quasi-isometric classification for all known cases. These rigidity results are mainly a consequence of the rigid structure of the intersection pattern of flats in the universal cover of the Salvetti complex.
	
	In this paper we describe the commensurability classes of RAAGs defined by trees of diameter at most 4 and describe the ``minimal'' group in each commensurability class (minimal in terms of number of generators or the rank of its abelianization). In particular, we show that there exist infinitely many different commensurability classes confirming a conjecture of Behrstock and Neumann. In their paper \cite{BN}, the authors show that RAAGs defined by trees of diameter at least 3 are quasi-isometric, so we provide first examples of RAAGs that are quasi-isometric but not commensurable. As in the classical case of lattices in locally compact topological groups, we define an ordered set that plays the role of the covolume and prove that the groups are commensurable if and only if the sets are commensurable, that is they have the same cardinality and constant rational ordered ratios.

	\subsection{Main results}
	
	Let $\Delta=(V(\Delta),E(\Delta))$ be a simplicial graph, then we denote by $\mathbb{G}(\Delta)$ the RAAG defined by the commutation graph $\Delta$.
	We call the vertices of the graph $\Delta$ the {\it canonical generators} of $\mathbb{G}(\Delta)$.
	
	For our purposes, it will be convenient to encode finite trees of diameter four as follows. Let $T$ be any finite tree of diameter four.  Let $f$ be a path (without backtracking) of length four from one leaf of $T$ to another. By definition $f$ contains 5 vertices and let $c_f\in V(T)$ be the middle vertex in $f$. It is immediate to see that the choice of the vertex $c=c_f$ does not depend on the choice of the path $f$ of length four. We call $c$ the \emph{center} of $T$.
	
	Any leaf of $T$ connected to $c$ by an edge is called a \emph{hair vertex}. Vertices connected to $c$ by an edge which are not hair are called \emph{pivots}.  Any finite tree $T$ of diameter 4 is uniquely defined by the number $q$ of hair vertices and by the number $k_i$ of pivots of a given degree $d_i+1$.  Hence we encode any finite tree of diameter $4$ as $T((d_1,k_1), \dots, (d_l,k_l);q)$. Here all $d_i$ and $k_i$ and $l$ are positive integers, $d_1 < d_2 < \ldots < d_l$, and $q$ is a non-negative integer; moreover, either $l \geq 2$ or $l=1$ and $k_1 \geq 2$, so that $T$ indeed has diameter 4. See Figure \ref{fig:tree4}.

	\begin{figure}[!h]
		\centering
		\includegraphics[keepaspectratio,width=3.2in]{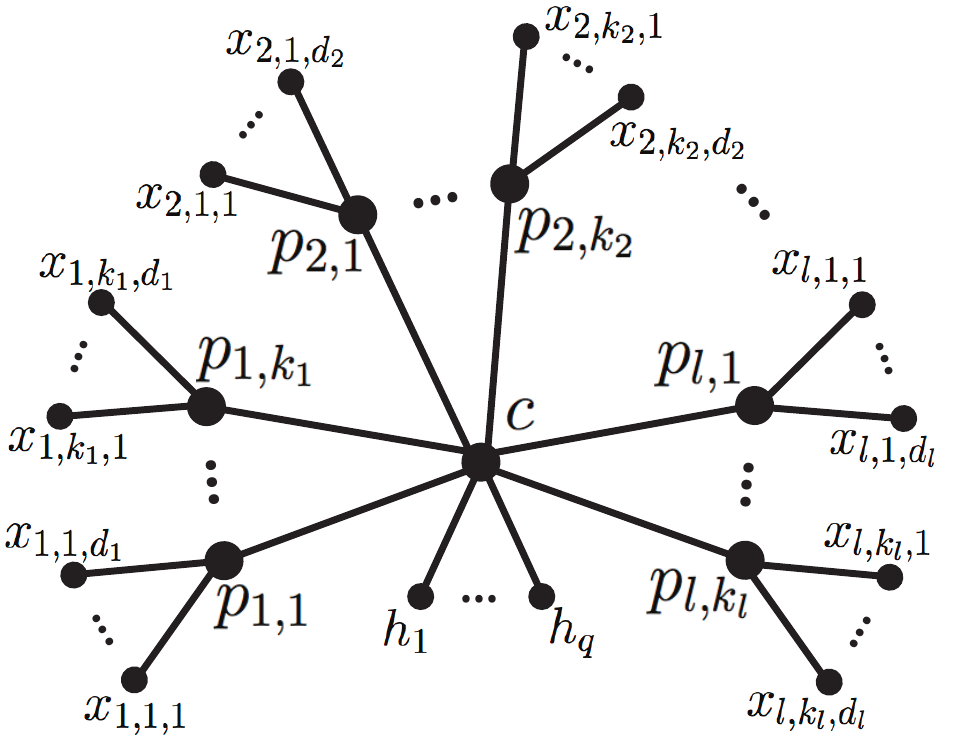}
		\caption{Tree of diameter 4} \label{fig:tree4}
	\end{figure}
	
	Given a tree of diameter four $T=T((d_1,k_1), \dots, (d_l,k_l);q)$, we denote by $M(T)=M(T((d_1,k_1), \dots, (d_l,k_l);q))$ the set of numbers $d_i$, that is,  we have that $M(T((d_1,k_1), \dots, (d_l,k_l);q))=\{d_1 < d_2 < \dots <d_l \}$.

	In the original usage two real numbers $a,b \in \mathbb{R}$ are commensurable if and only if the ratio $\frac{a}{b}$ is rational. In this fashion, we will say that two ordered sets $P=\{p_1 < \dots < p_k\}$ and $Q=\{q_1 < \dots < q_l\}$ are \emph{commensurable} if
	\begin{itemize}
		\item they have the same cardinality, i.e. $k=l$,
		\item there exists $c\in \mathbb{Q}$ such that each quotient $\frac{p_i}{q_i} = c$.
	\end{itemize}
		 In this case we write $P=cQ$.
	
		We say that a set $M=\{m_1, \dots, m_k\}$, $m_i \in \BN$, is {\it minimal} if the greatest common divisor $d=\gcd(m_1, \dots, m_k)$ is 1. It is clear that for each commensurability class of a set $M \subset \BN^k$, there exists a minimal set that belongs to the class, namely $\{\frac{m_1}{d}, \dots, \frac{m_k}{d}\}$.
	
	We show that the commensurability class of the set $M(T((d_1,k_1), \dots, (d_l,k_l);q))$ determines the commensurability class of the RAAG defined by the tree of diameter 4.

  \begin{theorem4}[Characterisation of commensurability classes]
		Let $T$ and $T'$ be two finite trees of diameter 4,  $T=T((d_1,k_1), \dots, (d_l,k_l);q)$ and $T'=T((d_1',k_1'), \dots, (d_{l'}',k_{l'}');q')$. Let $\GG=\mathbb{G}(T)$ and $\GG'=\mathbb{G}(T')$. Consider the sets $M=M(T)$, $M'=M(T')$.
		Then $\GG$ and $\GG'$ are commensurable if and only if $M$ and $M'$ are commensurable. 	
	\end{theorem4}
	
For $n > 1$ denote by $P_n$ the path with $n$ vertices and $n-1$ edges.	By a minimal RAAG in some class $C$ we mean a RAAG in $C$ with the minimal number of generators, i.e., defined by a graph with the minimal number of vertices among all commutation graphs of RAAGs in $C$.
	
	\begin{theorem3}
		Let $T=T((d_1,k_1), \dots, (d_l,k_l);q)$ be a finite tree of diameter 4. Let $\mathcal C(T)$ be the commensurability class of $\GG(T)$ and let $M=M(T)$ be as above, so $|M|=l$. Then the minimal RAAG that belongs to $\mathcal C(T)$ is either the RAAG defined by the tree $T'=T((d_1',1), \dots, (d_l',1);0)$, where $M(T')$ is minimal in the commensurability class of $M$, if $|M|>1$, or the RAAG defined by the path of diameter 3, that is $\GG(P_4)$, if $|M|=1$.
	\end{theorem3}

	Our results extend naturally to the commensurability classification of some right-angled Coxeter groups. Recall that every RAAG $\GG$ embeds naturally as a finite index subgroup into a right-angled Coxeter group, say $C(\GG)$, see \cite{DJ}. Hence, we have the following result
	\begin{cor*}
		There are infinitely many pair-wise quasi-isometric, but pair-wise not commensurable right-angled Coxeter groups defined by graphs of diameter 4 with cliques of dimension $2$.
	\end{cor*}

\subsection{Strategy of the proof}

As we discussed, in previous results on commensurability of RAAGs the structure of the intersection pattern of flats inside the universal cover of the Salvetti complex is so rigid that the large-scale geometry that it determines forces commensurability, see \cite{KKi, Huang}.
	
In our case, all trees of diameter 4 are quasi-isometric and so geometry is not sufficient to determine the commensurability classes. However, we will use the structure of the intersection pattern of flats together with algebra to derive the result.

In broad strokes, the strategy is as follows. To two given RAAGs $\GG(T)$ and $\GG(T')$ defined by trees of diameter 4, we associate a linear system of equations $S(T,T')$ and show that, if $\GG(T)$ and $\GG(T')$ are commensurable, then the system $S(T,T')$ has positive integer solutions, see Section \ref{sec:lineareq}. We then study the system $S(T,T')$ and determine conditions on the trees $T$ and $T'$ for which the system does not have positive integer solutions. This allows us to conclude, that the corresponding RAAGs are not commensurable, see Sections \ref{Sec:InfCom} and \ref{sec:noncomm}. In order to obtain a characterisation, we prove that if the conditions are not satisfied (and so the system $S(T,T')$ has positive integer solutions), then we can exhibit isomorphic finite index subgroups $H<\GG(T)$ and $H'<\GG(T')$ and conclude that $\GG(T)$ and $\GG(T')$ are commensurable, see Section \ref{Sec:com}.

We believe that the general strategy of our proof, i.e. to reduce the existence of subgroups to a linear system of equations with positive integer solutions, can be used to study commensurability classes of RAAGs defined by trees and more general RAAGs. The real obstacle is to elucidate the necessary conditions for the system of equations to have positive integer solutions -- ``just'' a linear algebra problem.
	
\section{Basics on RAAGs} \label{sec:prelim}

In this section we recall some preliminary results on RAAGs and introduce the notation we use throughout the text.

Let $\Gamma=(V(\Gamma), E(\Gamma))$ be a finite simplicial graph (i.e., a graph without loops and multiple edges) with vertex set $V(\Gamma)$ and edge set $E(\Gamma)$. Then, the {\it right-angled Artin group} (or {\it RAAG} for short) $\GG=\GG(\Gamma)$ defined by the (commutation) graph $\Gamma$ is the group given by the following presentation:
$$
\GG=\langle V(\Gamma)\mid [v_1,v_2]=1, \hbox{ whenever } (v_1,v_2)\in E(\Gamma)\rangle.
$$
The elements of $V(\Gamma)$ are called the {\it canonical generators} of $\GG$.

Let $\Gamma'=(V(\Gamma'), E(\Gamma'))$ be a full subgraph of $\Gamma$. It is not hard to show, see for instance \cite{EKR}, that the RAAG $\GG'=\GG(\Gamma')$ is the subgroup  of $\GG$ generated by $V(\Gamma')$, i.e. $\GG(\Gamma')=\langle V(\Gamma')\rangle$.

Let $X=V(\Gamma)$, and $u$ be a word in the alphabet $X \cup X^{-1}$. Denote by $[u]$ the element of $\GG$ corresponding to $u$, we also say that the word $u$ represents $[u] \in \GG$.
We denote the length of a word $u$ by $|u|$. A word $u$ is called {\it geodesic} if it has minimal length among all the words representing the same element $[u]$ of $\GG$. In RAAGs, any word can be transformed to a geodesic (representing the same element of $\GG$) by applying only free cancellations and permutations of letters allowed by the commutativity relations of $\GG$. Moreover, any two geodesic words representing the same element of $\GG$ can be transformed to each other by applying only the commutativity relations of $\GG$. See \cite{EKR} for details.

Let $w \in \GG$, and $u$ be any geodesic word representing $w$. Then the length of $w$ is defined to be the length of $u$, $|w|=|u|$.
   An element $w\in \GG$ is called \emph{cyclically reduced} if $|w^2|=2|w|$,
    or, equivalently, the length of $w$ is minimal in the conjugacy class of $w$. Every element is conjugate to a cyclically reduced one. 

We say that a letter $x$ of $X$ occurs in a word $u$ if at least one of the letters in $u$ is either $x$ or $x^{-1}$. For a given element $w \in \GG$, denote by $\az(w)$ the set of letters of $X$ occurring in $u$, where $u$ is any geodesic word representing $w$ (this does not depend on the choice of $u$, due to the above remarks). Also define $\BA(w)$ to be the subgroup of $\GG$ generated by all the letters in $X$ that do not occur in a geodesic word $u$ (which represents $w$) and commute with $w$.  Again, the subgroup $\BA(w)$ is well-defined (independent of the choice of $u$), due to the above remarks. Note also that a letter $x$ commutes with $w$ if and only if $x$ commutes with every letter in $\az(w)$, see \cite{EKR}.

In this paper we always conjugate as follows: $g^h=hgh^{-1}$. 		

	
For a RAAG $\GG(\Gamma)$ we define its non-commutation graph $\Delta=(V(\Delta), E(\Delta))$ as follows: $V(\Delta)=V(\Gamma)$ and $E(\Delta)=(V(\Gamma)\times V(\Gamma)) \setminus E(\Gamma)$, i.e., $\Delta$ is the complement graph of $\Gamma$. The graph $\Delta$ is a union of its connected components $I_1, \ldots , I_k$, which induce a decomposition of $\GG$ as the direct product
	$$
	\GG= \GG(I_1) \times \cdots \times \GG(I_k).
	$$
	
	Given a cyclically reduced $w \in \GG$ and the set $\az(w)$, consider the graph $\Delta (\az(w))$, which is the full subgraph of $\Delta$ with the vertex set consisting of letters in $\az(w)$. If the graph is connected, we call $w$ a \emph{block}. If $\Delta(\az(w))$ is not connected, then we can decompose $w$ into the product
	\begin{equation} \label{eq:bl}
	w= w_{j_1} \cdot w_{j_2} \cdots w_{j_t};\ j_1, \dots, j_t \in J,
	\end{equation}
	where $|J|$ is the number of connected components of $\Delta(\az(w))$ and the word $w_{j_i}$ is a word in the letters from the $j_i$-th connected component. Clearly, the words $\{w_{j_1}, \dots, w_{j_t}\}$ pairwise commute. Each word $w_{j_i}$, $i \in {1, \dots,t}$, is a block and so we refer to expression (\ref{eq:bl}) as the {\it block decomposition} of $w$.

	An element $w\in \GG$ is called a \emph{least root} (or simply, root) of $v\in \GG$ if there exists a positive integer $1 \leq m\in \BN$ such that $v=w^m$ and there does not exist $w'\in \GG$ and $1< m'\in \BN$ such that $w={w'}^{m'}$. In this case, we write $w=\sqrt{v}$. By a result from \cite{DK}, RAAGs have least roots, that is the root element of $v$ is defined uniquely for every $v \in \GG$.
	
	The next result describes centralisers of elements in RAAGs. Since centralizers of conjugate elements are conjugate subgroups, it suffices to describe centralizers of cyclically reduced elements.
	
	\begin{theorem}[Centraliser Theorem, Theorem 3.10,  \cite{Serv} and \cite{DK}] \label{thm:centr} \
		Let $w\in \GG$ be cyclically reduced and $w=v_1\dots v_k$ be its block decomposition. Then, the centraliser of $w$ is the following subgroup of $\GG$:
		\begin{equation} \notag
		C(w)=\langle \sqrt{v_1}\rangle \times \cdots \times \langle \sqrt{v_k} \rangle\times \BA(w).
		\end{equation}
	\end{theorem}
	The following two corollaries follow immediately from Theorem \ref{thm:centr} and the definitions.
	\begin{corollary} \label{cor:centr}
		For any $w\in \GG$ we have $C(w)=C(\sqrt{w})$.
	\end{corollary}

	\begin{corollary}\label{cor:123}
		For any vertex $v$ of $\Gamma$ the centralizer of $v$ in $\GG(\Gamma)$ is generated by all the vertices in the star of $v$ in $\Gamma$, i.e., by all the vertices adjacent to $v$ and $v$ itself. \\ In particular, if $\Gamma$ is a tree, then the centralizer of any vertex in $\Gamma$ is isomorphic to $\mathbb{Z} \times F_n$, where $n$ is the degree of $v$, so it is isomorphic to $\mathbb{Z}^2$ if $v$ is a leaf, and contains a non-abelian free group otherwise.
		
	\end{corollary}
The following corollary will play a key role in this paper.	
	
	\begin{corollary}\label{cor:centraliserstrees}
			Let $\Gamma$ be a tree and let $w\in \GG(\Gamma)$. Then $C(w)$ is non-abelian if and only if $w$ is conjugate to a power of a non-leaf vertex generator of $\GG(\Gamma)$, and in this case $C(w)\simeq \BZ\times F=\langle \sqrt{w}\rangle \times F$, where $F$ is a non-abelian free group.
	\end{corollary}
	\begin{proof}
		One implication follows immediately from Corollaries \ref{cor:centr} and \ref{cor:123}. For the other implication, suppose that $w$ has non-abelian centralizer, and let $w=w_0^g$, where $w_0$ is cyclically reduced. Since $C(w_0)$ is non-abelian, it follows from Theorem \ref{thm:centr} that $\BA(w_0)$ is not cyclic, and since $\Gamma$ is a tree this implies that $w_0$ is a power of some non-leaf generator, by definition of $\BA(w_0)$.
	\end{proof}
	
	\begin{corollary} \label{cor:cent}
			Centralizers of elements in RAAGs are RAAGs themselves.
	\end{corollary}
	\vspace{-0.3 cm}
	\begin{proof}
		Follows immediately from Theorem \ref{thm:centr} and the fact that $\BA(w)$ is a RAAG by definition.
	\end{proof}
	Note that in general RAAGs contain lots of subgroups which are not RAAGs themselves, even of finite index.
	
\section{Necessary conditions for commensurability}
	
\subsection{(Reduced) Extension graph and centraliser splitting}

In this section we recall the notions of the (reduced) extension graph and the (reduced) centraliser splitting. In the particular case when the underlying graph $\Delta$ is a tree, so is the (reduced) extension graph, see \cite{KK}. The goal of this section is to show that there exists an equivariant isomorphism between the (reduced) extension graph and the Bass-Serre tree of the (reduced) centraliser splitting of $\GG(\Delta)$.
	
\begin{defn}[Extension graph, see \cite{KK}]
Let $\GG(\Delta)$ be a RAAG with underlying commutation graph $\Delta$, then the {\it extension graph} $\Delta^e$ is defined as follows. The vertex set of $\Delta^e$ is the set of all elements of $\mathbb{G}(\Delta)$ which are conjugate to the canonical generators (vertices of $\Delta$). Two vertices are joined by an edge if and only if the corresponding group elements commute. The group $\mathbb{G}(\Delta)$ acts on $\Delta^e$ by conjugation.
\end{defn}
	
In some sense, the extension graph encodes the structure of the intersection pattern of flats inside the universal cover of the Salvetti complex. It plays an essential role in establishing the quasi-isometric rigidity of the class of RAAGs with finite outer automorphism groups, see \cite{Huang}.
	
	\medskip

Observe that RAAGs split as fundamental groups of graph of groups, whose vertex groups are centralisers of vertex generators. In this paper we will work with the centraliser splitting defined as follows.
\begin{defn}[(Reduced) Centraliser splitting]
Let $\Delta$ be a tree and let $\GG(\Delta)$ be the RAAG with underlying graph $\Delta$. The {\it centraliser splitting} of $\GG(\Delta)$ is a graph of groups defined as follows. The graph of the splitting is isomorphic to $\Delta$ and the vertex group at every vertex is defined to be the centralizer of the corresponding vertex generator. Note that if $v$ is some vertex of $\Delta$, and $u_1, \ldots, u_s$ are all vertices of $\Delta$ adjacent to $v$, then $C(v)=\langle v,u_1,\ldots,u_s \rangle \cong \mathbb{Z} \times F_s$, where $F_s$ is the free group of rank $s$, see Corollary \ref{cor:123}. In particular, $C(v)$ is abelian if and only if $v$ has degree 1, and in this case $C(v) \cong \mathbb{Z}^2$ is contained in the centralizer of the vertex adjacent to $v$. For an edge $e$ connecting vertices $u$ and $v$ the edge group at $e$ is $C(u) \cap C(v)=\langle u,v \rangle  \cong \mathbb{Z}^2$.

Note that the centralizer splitting is not reduced, since for every vertex of degree 1 in $\Delta$ the vertex group is equal to the incident edge group. Thus it makes sense to consider the {\it reduced centraliser splitting} of $\mathbb{G}(\Delta)$ (for a tree $\Delta$), which is obtained from the centralizer splitting by removing all vertices of degree 1. In this splitting all the vertex groups are non-abelian, and all the edge groups are isomorphic to $\mathbb{Z}^2$, in particular, this splitting is already reduced.
\end{defn}	

As we show in Lemma \ref{bs} below, just as the centraliser splitting corresponds to the extension graph, the reduced centraliser splitting corresponds to the reduced  extension graph, which we now define.

\begin{defn}[Reduced extension graph]
For a tree $\Delta$, we define the {\it reduced extension graph} of $\Delta$, and denote it by $\widetilde{\Delta}^e$,  to be the full subgraph of the extension graph $\Delta^e$, whose vertex set is the set of all elements of $\mathbb{G}(\Delta)$ which are conjugate to the canonical generators corresponding to vertices of $\Delta$ of degree more than 1 (which are exactly those which have non-abelian centralizers).
\end{defn}	
	
The reduced extension graph will play an essential role in the classification of RAAGs up to commensurability. If $H< \GG(\Delta)$ is a subgroup of finite index, then $H$ intersects each cyclic subgroup associated to the vertex groups of the reduced extension graph in a non-trivial cyclic subgroup, that is $H \cap \langle a_i^g\rangle  = \langle {(a_i^{k_i})}^g\rangle$ for some $k_i \in \BN$, $g\in \GG$. On the other hand, in the case of trees, by the description of centralisers in RAAGs, every element in $H$ whose centraliser is non-abelian belongs to some cyclic subgroup $\langle {(a_i^{k_i})}^g\rangle = H \cap \langle a_i^g\rangle$, see Corollary \ref{cor:centraliserstrees}. Since the set of elements with non-abelian centralisers is an invariant set up to isomorphism, it follows that the reduced extension graph is an algebraic invariant in the class of finite index subgroups of $\GG(\Delta)$, that is if $K \simeq H$ and $H<_{fi} \GG(\Delta)$, then the graph $T(K)$ whose vertex set is in one-to-one correspondence with maximal cyclic subgroups generated by elements of $K$ with non-abelian centraliser and there is an edge $(u,v)$ whenever the corresponding elements commute, is isomorphic to the reduced extension graph of $\GG(\Delta)$ (see Lemma \ref{l2}).

From this observation, one can deduce that many classes of RAAGs are not commensurable (without using the stronger fact that they are not quasi-isometric). For instance, RAAGs whose defining graphs are trees are not commensurable to RAAGs whose defining graph have cycles; or RAAGs defined by cycles of different lengths are not commensurable, etc. At this point, a couple of remarks are in order. Firstly, the above observation does not extend to the extension graph, that is, commensurable RAAGs may not have isomorphic extension graphs. For instance, we will show that the RAAGs defined by paths of length 3 and 4 are commensurable but their extension graphs are not isomorphic (leaves in the tree are also leaves in the extension graphs and the minimal distances in the extension graphs between leaves is 3 and 4 respectively). The reason here is that there are elements with abelian centralisers that are \emph{not} powers of conjugates of the canonical generators so they are not accounted for in the extension graph. Secondly, the isomorphism of the reduced extension graphs is a necessary condition but it is by far not sufficient. All RAAGs defined by trees of diameter 4 have isomorphic reduced extension graphs but there are infinitely many different commensurability classes among them.
	
	\bigskip

	In the next lemmas we notice that the two trees, the (reduced) extension graph and the Bass-Serre tree associated to the (reduced) centraliser splitting are equivalent and that the reduced extension graph and its quotient by the action of $H$ is invariant up to isomorphism.

	\begin{lemma}\label{bs}
		If $\Delta$ is a tree, then the extension graph $\Delta^e$ is also a tree, which is isomorphic to the Bass-Serre tree $T$ corresponding to the centralizer splitting of $\mathbb{G}(\Delta)$.
		Moreover, the reduced extension graph $\widetilde{\Delta}^e$ is a subtree of $\Delta^e$, which is isomorphic to the Bass-Serre tree $\widetilde{T}$ corresponding to the reduced centralizer splitting of $\mathbb{G}(\Delta)$. The graph isomorphisms above are equivariant, in the sense that the action of $\mathbb{G}(\Delta)$ by conjugation on $\Delta^e$ {\rm(}or $\widetilde{\Delta}^e${\rm)} corresponds to the natural action of $\mathbb{G}(\Delta)$ on the Bass-Serre tree of the centralizer splitting {\rm(}reduced centralizer splitting, respectively{\rm)}.
	\end{lemma}
	From now on we denote $\GG=\mathbb{G}(\Delta)$.
	\begin{proof}
		Every vertex of $\Delta^e$ has the form $v^g$, where $v$ is some canonical generator of $\GG$, and $g \in \GG$. By Bass-Serre theory, vertices of $T$ correspond to left cosets of centralizers of canonical generators of $\GG$. Define a morphism $\beta$ from $\Delta^e$ to $T$ by sending $v^g$ to the vertex of the form $gC(v)$. This gives a bijection of the vertex sets. By definition, two vertices $v_1^{g_1}$ and $v_2^{g_2}$ of $\Delta^e$ are connected by an edge iff $[v^{g_1}_1,v^{g_2}_2]=1$ iff
		$$
		[v_1,v_2]=1 \hbox{ and } v^{g_1}_1=v_1^g, v^{g_2}_2=v_2^g \hbox{ for some } g \in \GG,
		$$
		iff
		$$
		[v_1,v_2]=1 \hbox{ and } g_1^{-1}g \in C(v_1), \: g_2^{-1}g \in C(v_2) \hbox{ for some } g \in \GG,
		$$
		iff
		$$
		[v_1,v_2]=1\hbox{ and } g_1C(v_1) \cap g_2C(v_2) \neq \emptyset.
		$$
		By Bass-Serre theory, since $\Delta$ is a tree and therefore there are no HNN-extensions appearing, two vertices $g_1C(v_1)$ and $g_2C(v_2)$ of $T$ are connected by an edge iff $[v_1,v_2]=1$ and $gC(v_1)=g_1C(v_1), \: gC(v_2)=g_2C(v_2)$ for some $g \in \GG$, iff $[v_1,v_2]=1$ and $g_1C(v_1) \cap g_2C(v_2) \neq \emptyset$. This shows that $\beta$ is indeed a graph isomorphism. Now $\beta(v^g)=gC(v)=g\beta(v)$, so $\beta$ is equivariant, and the restriction of $\beta$ to $\widetilde{\Delta}^e$ gives an equivariant isomorphism between $\widetilde{\Delta}^e$ and $\widetilde{T}$.
	\end{proof}

\subsection{Commensurability invariants: the reduced extension graph and the quotient graph}	
	
	The goal of this section is to show that the reduced extension graph and the quotient graph (defined below) are commensurability invariants. Suppose $\Delta$ is a finite tree, and $K$ is a finite index subgroup of $\mathbb{G}(\Delta)$. Then $\mathbb{G}(\Delta)$ acts on the reduced extension graph $\widetilde{\Delta}^e$, and $K$ acts on $\widetilde{\Delta}^e$ by restriction. Let $\Psi(K)$ be the quotient graph: $\Psi(K)=K \backslash  \widetilde{\Delta}^e$. The goal of this section is to show that if $\GG(\Delta)$ and $\GG(\Delta')$ are commensurable, and $\GG(\Delta)>_{fi}K\simeq K'<_{fi}\GG(\Delta')$, then the reduced extension graphs $\widetilde{\Delta}^e$ and $\widetilde{\Delta}'^e$ are isomorphic and so are the quotient graphs $\Psi(K)=K\backslash\widetilde{\Delta}^e$ and $\Psi(K')=K'\backslash\widetilde{\Delta}'^e$.

	Note that there are natural projection graph morphisms
	$$
	\gamma_K: \widetilde{\Delta}^e \rightarrow \Psi(K)
	$$
	and
	$$
	\delta_K: \Psi(K) \rightarrow \widetilde{\Delta}=\GG(\Delta)\backslash\widetilde{\Delta}^e.
	$$
	Note also that the graph $\Psi(K)$ is finite, since $K$ has finite index in $\mathbb{G}(\Delta)$.
	
	Below $C_K(w)=K \cap C_{\GG}(w)$ is defined even if $w \in \GG$ is not in $K$. In fact, since $K$ has finite index in $\GG$, there exists a natural $m$ such that $w^m \in K$, and then $C_K(w)=C_K(w^m)$.
	
	\begin{lemma}\label{l3}
		Let $v$ be a vertex of $\Psi(K)$, and $w$ be any vertex of $\widetilde{\Delta}^e$ such that $\gamma_K(w)=v$. 
		 Then
		\begin{enumerate}
			\item There is a bijection between the vertices of $\Psi(K)$ adjacent to $v$ and the equivalence classes of vertices of $\widetilde{\Delta}^e$ adjacent to $w$, modulo conjugation by $K$.
			\item There is a bijection between the edges of $\Psi(K)$ incident to $v$ and the equivalence classes of vertices of $\widetilde{\Delta}^e$ adjacent to $w$, modulo conjugation by $C_K(w)$.
		\end{enumerate}
	\end{lemma}	
	\begin{proof}
		The first claim follows immediately from the definition of $\Psi(K)$ as a quotient graph.
		Now let $e_1$ and $e_2$ be two different edges of $\widetilde{\Delta}^e$ incident to $w$. Let $w_1$, $w_2$ be the other ends of $e_1$, $e_2$ respectively. Then $e_1$ and $e_2$ project into the same edge of $\Psi(K)$ if and only if there exists $g \in K$ which takes $e_1$ to $e_2$. Since $w$ and $w_1$ for sure belong to different orbits under the action of $K$ (even of $\GG$), this happens if and only if $g$ takes $w_1$ to $w_2$ and leaves $w$ fixed. This means that $w_1^g=w_2$ and $w^g=w$, so $g \in C_K(w)$, thus the second claim also holds.
	\end{proof}	
	In particular, graph $\Psi(K)$ can have multiple edges and cycles (but not loops).	
	
	Suppose now $\Delta$ and $\Delta'$ are finite trees, and denote $\GG=\mathbb{G}(\Delta)$ and $\GG'=\mathbb{G}(\Delta')$. Suppose $\GG$ and $\GG'$ are commensurable.
	Thus there exist finite index subgroups $H \le \GG$ and $H' \le \GG'$ and an isomorphism $\varphi: H \rightarrow H'$.
	
	Consider the reduced extension graphs $\widetilde{\Delta}^e$ and $\widetilde{\Delta}'^e$ defined above, with the actions by conjugation of $\GG$ and $\GG'$ respectively.
	The finite graphs $\Psi(H)$ and $\Psi(H')$ are defined as above.
	
	\begin{lemma}\label{l2}
		The group isomorphism $\varphi: H \rightarrow H'$ induces graph isomorphisms $\overline{\varphi}: \widetilde{\Delta}^e \rightarrow \widetilde{\Delta}'^e$ and $\varphi_*: \Psi(H) \rightarrow \Psi(H')$.
	\end{lemma}
	\begin{proof} Let $u_1, \ldots, u_k$ be all vertices of $\Delta$ which have degree more than 1, and $u_1', \ldots, u_{k'}'$ be all vertices of $\Delta'$ which have degree more than 1. Then  $u_1, \ldots, u_k$ are all canonical generators of $\GG$ with non-abelian centralizers, and $u_1', \ldots, u_{k'}'$ are all canonical generators of $\GG'$ with non-abelian centralizers.
	
		Notice that the only elements of $\GG$ which have non-abelian centralizers in $\GG$ are conjugates of powers of $u_1, \ldots, u_k$ by some element of $\GG$. Since $H$ has finite index in $\GG$ and so $C_H(x)=H \cap C_\GG(x)$ has finite index in $C_\GG(x)$ for any $x \in \GG$. The only elements of $H$ which have non-abelian centralizers in $H$ are conjugates of powers of $u_1, \ldots, u_k$ by some element of $\GG$ which belong to $H$, that is elements from the set $M_H=\{(u_i^{k_i})^{g_i}\in H\}$, where $g\in \GG$, $k_i\in \BZ$. Analogously, the only elements of $H'$ which have non-abelian centralizers in $H'$ are conjugates of powers of $u_1', \ldots, u_{k'}'$ by some element of $\GG'$ which belong to $H'$, that is elements from the set $M_{H'}=\{({u_i'}^{k_i'})^{g_i'}\in H'\}$, where $g\in \GG'$, $k_i'\in \BZ$. Thus, the isomorphism $\varphi$ should take the set $M_H$ to the set $M_{H'}$.
		
		Now define a morphism of trees $\overline{\varphi}: \widetilde{\Delta}^e \rightarrow\widetilde{\Delta}'^e$ as follows. First define $\overline{\varphi}$ on the vertex set of $\widetilde{\Delta}^e$. Let $w$ be a vertex of $\widetilde{\Delta}^e$. Then $w=u^g$ for some $g \in \GG$, where $u$ is one of $u_1, \ldots, u_k$. Since $H$ has finite index in $\GG$, there exists a minimal positive integer $k$ such that $w^k \in H$. Then $w^k$ has non-abelian centralizer in $H$, so $\varphi(w^k)$ also has non-abelian centralizer in $H'$, thus $\varphi(w^k)=(u_0^l)^{g_0} \in H'$ for some positive integer $l$, $g_0 \in \GG'$, and $u_0$ equal to one of $u_1', \ldots, u_{k'}'$. Denote $w_0=u_0^{g_0}$.  Then let $\overline{\varphi}(w)=w_0$. Note that here $l$ is also equal to the minimal positive integer $i$ such that $w_0^i \in H'$, since if $i < l$, then also $i | l$, so $w_0^l$ is a proper power in $H'$, but this is impossible, since $w^k$ is not a proper power in $H$, and  $w_0^l=\varphi(w^k)$.
		
		We can extend $\overline{\varphi}$ to the edges in a natural way. Indeed, by description of centralisers in RAAGs, see Theorem \ref{thm:centr}, two vertices $w_1, w_2$ of $\widetilde{\Delta}^e$ are adjacent iff
		$$
		[w_1,w_2]=1 \Leftrightarrow [w_1^{k_1},w_2^{k_2}]=1
		$$
		(take positive integers $k_1, k_2$ such that $w_1^{k_1},w_2^{k_2} \in H$), iff
		$$
		[\varphi(w_1^{k_1}),\varphi(w_2^{k_2})]=1\Leftrightarrow [\overline{\varphi}(w_1),\overline{\varphi}(w_2)]=1,
		$$
		iff $\overline{\varphi}(w_1)$ and $\overline{\varphi}(w_2)$ are adjacent. This shows that $\overline{\varphi}$ is a well-defined morphism of trees. Moreover, $\overline{\varphi}$ is in fact an isomorphism of trees, since the inverse morphism can be constructed in the same way.

		If $w$ is a vertex of $\widetilde{\Delta}^e$, and $h \in H$, then we have that $\overline{\varphi}(w^h)=\overline{\varphi}(w)^{\varphi(h)}$. It follows that we can restrict $\overline{\varphi}$ to the graph isomorphism
		$$
		\varphi_*: \Psi(H)= H \backslash  \widetilde{\Delta}^e \rightarrow \Psi(H')=  H' \backslash  \widetilde{\Delta}^e, \quad \varphi_*(\gamma_H(z))=\gamma_{H'}(\overline{\varphi}(z)),
		$$
		where $z$ is a vertex or an edge of $\widetilde{\Delta}^e$, and $\gamma_H$, $\gamma_{H'}$ are the orbit projections as above. This proves the lemma.
	\end{proof}
	
	\subsection{Commensurability invariant: linear relations between minimal exponents}\label{sec:lineareq}
	
	So far we have observed that isomorphisms leave the set of powers of conjugates of canonical generators with non-abelian centralisers invariant. However, if an isomorphism $\varphi$ sends ${(a^k)}^g$  to ${(a'^{k'})}^{g'}$, then a priori there is no relation between the integer numbers $k$ and $k'$ corresponding to the powers.
	
	\begin{expl}
	 Let $H, K<_{fi} \BZ \times F_2 \simeq \langle c \rangle \times \langle a,b \rangle$, $H= \langle c^k,a,b \rangle$ and $K=\langle c^{k'}, a,b\rangle$. Clearly, $\varphi:H \to K$ that maps $c^k \to c^{k'}$, $a\to a$ and $b\to b$ is an isomorphism and $k$ and $k'$ can be taken arbitrarily.
	
	On the other hand, if we consider $H=\langle c^k, a^2,b,a^b\rangle$, $K=\langle c^{k'}\rangle \times F_m <_{fi}\langle c \rangle \times \langle a,b \rangle$ and we \emph{assume} that the isomorphism $\varphi:H\to K$ sends $c^k$ to $c^{k'}$ and each power of a conjugate of either $a$ or $b$ to a power of a conjugate of $a$ and $b$, then we do get a constraint on the possible powers of the image of $a$. Indeed, the isomorphism $\varphi$ induces an isomorphism $\varphi'$ from the subgroup $\langle a^2,b,a^b\rangle$ to $F_m$, hence $m=3$ and the index of $F_3 < F_2$ is $2$. Since among the conjugates of powers of generators in $H$ there are $3$ of minimal exponent, by our assumption, there are also $3$ in $K$. Furthermore, since the sum of the minimal exponents of conjugates of a fixed generator is equal to the index, which in our case is $2$, and since, upto relabeling, there are only two covers of $F_2$ of degree $2$ and only one of them has $3$ conjugates of powers of generators of minimal exponent, it follows that either $a$ is sent to a conjugate of a generator (and there are exactly two conjugates of this generator) or it is sent to a square of a generator (and it is the only conjugate of minimal exponent of this generator in the subgroup).
 \end{expl}

	Our next goal is to formalise and generalise these ideas in order to find linear relations between the minimal exponents of the different (conjugacy) classes of powers of generators that belong to the subgroups $H$ and $H'$ correspondingly.
	
	The strategy is as follows. We already established that the isomorphism between finite index subgroups $H \le \GG$ and $H' \le \GG'$ induces an isomorphism between non-abelian centralisers in $H$ and $H'$: $C_H(w)\simeq C_{H'}(w')$, where $w$ and $w'$ are conjugates of generators of $\GG$ and $\GG'$ correspondingly. Since $H$ and $H'$ are of finite index in $\GG$ and $\GG'$, so are $C_H(w)$ and $C_{H'}(w')$ in the centralisers $C_1=C_{\GG}(w)$ and $C_2=C_{\GG'}(w')$ respectively. Note that $C_1$ and $C_2$ are of the form $\hbox{centre}\times \hbox{free group}$. Hence, the isomorphism between $C_H(w)$ and $C_{H'}(w')$ induces an isomorphism between the images $P_{H,w}$ and $P_{H',w'}$ of $C_H(w)$ and $C_{H'}(w')$ in the free groups obtained from $C_1$ and $C_2$ by killing their centres, see Lemma \ref{l4-1}.
	
	In general, for a conjugate $v=x^g$ of a canonical generator $x$ in $\GG$, such that $v \in C_{\GG}(w)$, the minimal exponent $k$ such that $v^{k}$ belongs to $C_H(w)$ does not coincide with the minimal exponent $l$ such that $\pi(v)^{l}$ belongs to $P_{H,w}$, where $\pi$ is the quotient by the center of $C_{\GG}(w)$ (i.e., by $\langle w \rangle$) homomorphism, see Example \ref{expl:minexp}. We refer to such $l$ as the minimal quotient exponent; it is equal to the minimal positive number $m$ such that $v^mw^s \in C_H(w)$ for some $s$, see Definition \ref{defn:minexps}.
	
	Since the ranks of the isomorphic subgroups $P_{H,w}$ and $P_{H',w'}$ of free groups are the same, the Schreier formula relates the corresponding indexes of $P_{H,w}$ and $P_{H',w'}$. We show in Lemma \ref{l4}, that the index of $P_{H,w}$ coincides with the sum of the minimal quotient exponents of conjugates of a given generator that belong to $P_{H,w}$.
	
	The goal of Section \ref{sec:minqexp} is to describe a linear relation between the minimal quotient exponents of conjugates of generators that belong to $P_{H,w}$ and $P_{H',w'}$ correspondingly, see Corollary \ref{cor:relminqexp}.
	
	In Section \ref{sec:minexp} we establish a relation between minimal exponents and minimal quotient exponents, see Lemma \ref{prop1}, and deduce a linear relation between the minimal exponents of the conjugates of generators that belong to $H$ and $H'$ correspondingly, see Lemma \ref{together}.
	
	\subsubsection{Linear relations between minimal quotient exponents} \label{sec:minqexp}

	Fix a finite index subgroup $K \le \GG$. We now encode the minimal exponent as a label of a vertex in the reduced extension graph $\widetilde{\Delta}^e$, and then of $\Psi(K)$.
	
	\begin{defn}[Label of a vertex/edge (minimal exponent/quotient exponent)] \ \label{defn:minexps}
	Let $w$ be a vertex of $\widetilde{\Delta}^e$, thus $w$ is also an element of $\GG$. Define the {\it label of the vertex $w$}, denoted by $\overline{L}(w)$, to be the minimal positive integer $k$ such that $w^k \in K$. Such number exists, since $K$ has finite index in $\GG$.
	
	Similarly, we encode the minimal quotient exponent as a label of an edge of reduced extension graph $\widetilde{\Delta}^e$  (and so of $\Psi(K)$) as follows. For an edge $f$ of $\widetilde{\Delta}^e$ connecting vertices $w_1$ and $w_2$ define the {\it label of the edge $f$ at the vertex $w_1$}, denoted by $\overline{l}_{w_1}(f)$, to be the minimal positive integer $k$ such that there exists an integer $l$ such that $w_1^k w_2^l \in K$. Note that we can always suppose $l$ is non-negative. Analogously the label of $f$ at $w_2$ is defined. Note that by definition $\overline L(w_1) \geq \overline{l}_{w_1}(f)$, for all edges $f$.
	\end{defn}
	
	Note that the labels of vertices and edges are invariant under the action of $K$ on $\widetilde{\Delta}^e$ (by conjugation). Indeed, for $h \in K$ $w^k \in K$ iff $(w^h)^k \in K$, and $w_1^k w_2^l \in K$ iff $(w_1^h)^k (w_2^h)^l \in K$.
	
	This means that we can define labels for the quotient graph $\Psi(K)$ as well. If $v$ is a vertex of $\Psi(K)$, then define the {\it label of the vertex $v$}, denoted by $L(v)$, to be the label $\overline{L}(w)$, where $w$ is some vertex of $\widetilde{\Delta}^e$ such that $\gamma_K(w)=v$. Analogously, if $p$ is an edge of $\Psi(K)$ connecting vertices $v_1$ and $v_2$, then define the {\it label of the edge $p$ at the vertex $v_1$}, denoted by $l_{v_1}(p)$, to be the label $\overline{l}_{w_1}(f)$, where $f$ is some edge of   $\widetilde{\Delta}^e$ such that $\gamma_K(f)=p$, and $w_1$ is the end of $p$ such that $\gamma_K(w_1)=v_1$. These labels are well-defined. Note that the labels of vertices and edges of $\Psi(K)$ are positive integers.
	
	\begin{expl}[Minimal exponent vs minimal quotient exponent]\label{expl:minexp}
	Consider the free group $F(a,b)$ of rank $2$ and consider an index $4$ subgroup $F_5=\langle b, b^a, b^{a^2}, b^{a^3}, a^4\rangle<F(a,b)$.

	Let $\GG=F(a,b)\times \langle c\rangle$ and let $K=\langle F_5, a^2c, c^2\rangle$ be a finite index subgroup of $\GG$. An easy computation shows that the minimal exponent of $a$ equals $4$,
	$$
	\min \{n\in \BN\mid a^n \in C_K(c^2)=K\}=4.
	$$
	Let $\pi:\GG\to \GG/\langle c\rangle$ be a natural projection. Then $\pi(K)=\langle \ov{b}, \ov{b}^{\ov{a}}, \ov{a}^2\rangle$, hence the minimal quotient exponent of $a$ equals $2$,
	$$
	\min \{n\in \BN\mid \ov{a}^n \in \pi(C_K(c^2))=\pi(K)\}=2.
	$$
	\end{expl}
	
	Let $w$ be a vertex of $\widetilde{\Delta}^e$, and let $\gamma_K(w)=v$ and $u=\delta_K(v)$. Thus $u$ is a canonical generator of $\GG$ with non-abelian centralizer, and $w=u^g$ for some $g \in \GG$.
	Suppose $u_1, \ldots, u_k$ are all vertices of $\Delta$ adjacent to $u$, then $k \geq 2$ is the degree of $u$ in $\Delta$, and $u_1, \ldots, u_s$ are those of them which have degree more than 1 (and so have non-abelian centralizer), here $s \leq k$. Note that
	$$
	C_K(w) \le C_\GG(w) \cong Z(C_\GG(w)) \times  F(u_1^g, \ldots, u_k^g),
	$$
	where $Z(C_\GG(w))= \langle w \rangle$ is cyclic and $F(u_1^g, \ldots, u_k^g)$ is the free group of rank $k$ with the basis $u_1^g, \ldots, u_k^g$. Let $F_{K,w}$ be the free group of rank $k$ with the basis $q_1, \ldots, q_k$. Denote by
	$$
	\pi:C_\GG(w) \rightarrow F_{K,w}
	$$
	the factorization by the center homomorphism induced by the map $\pi(w)=1, \: \pi(u_i^g)=q_i, \: i=1, \ldots, k$. Then $\pi$ induces an isomorphism between $F(u_1^g, \ldots, u_k^g)$ and $F_{K,w}$. Below, when we speak about cycles in the Schreier graph, we mean simple cycles with all edges labelled by the same generator, and we mean the Schreier graph with respect to the generators $q_1, \ldots, q_k$ of $F_{K,w}$.

	\begin{lemma}\label{l4-1}
	In the above notation, the epimorphism $\pi$ induces a group embedding $\pi_C:  C_K(w) / Z(C_K(w)) \hookrightarrow F_{K,w}$, and $P_{K,w}=\pi_C(C_K(w) / Z(C_K(w)))=\pi(C_K(w))$ is a finite index subgroup of $F_{K,w}$.
	\end{lemma}
	\begin{proof}
		Let $j$ be the minimal (positive) power of $w$ which belongs to $K$. Note that $Z(C_K(w))\le Z(C_\GG(w))$, since if some $z \in C_\GG(w)$, then for some positive integer $l$ $z^l \in C_K(w)$, so every element in the center of $C_K(w)$ should also commute with $z$. Also $Ker (\pi) =   Z(C_\GG(w))=  \langle w \rangle,$ so $Ker (\pi|_{C_K(w)})= K \cap  Z(C_\GG(w)) = \langle w^j \rangle = Z(C_K(w))$. This means that $\pi$ induces an embedding $\pi_C:  C_K(w) / Z(C_K(w)) \hookrightarrow F_{K,w}$. Since $C_K(w)=K \cap C_\GG(w)$ has finite index in $C_\GG(w)$, it follows that $P_{K,w}$ has finite index in $F_{K,w}$.
	\end{proof}

	\begin{lemma}\label{l4}
		In the above notation the following statements hold:
		\begin{enumerate}
			\item $\pi$ induces a bijection $\pi_*$ between the edges of $\Psi(K)$ incident to $v=\gamma_K(w)$ and those cycles in the Schreier graph of $P_{K,w}$ in $F_{K,w}$ which correspond to the generators $q_1, \ldots, q_s$.
			\item  For every edge $p$ of $\Psi(K)$ connecting $v=\gamma_K(w)$ with some other vertex $v_1$, the label $l_{v_1}(p)$ of $p$ with respect to $v_1$ is equal to the length of the cycle $\pi_*(p)$ in the Schreier graph of $P_{K,w}$ in $F_{K,w}$.
		\end{enumerate}
	\end{lemma}
	\begin{proof}
		Consider the Cayley graph of $F_{K,w}$ (with respect to the basis $q_1,\ldots,q_k$). A \emph{line} in the Cayley graph of $F_{K,w}$ is a bi-infinite path, where every edge is labelled by the same generator.
		
		We define a bijection $\overline{\pi}$ between the edges of $\widetilde{\Delta}^e$ incident to $w$ and lines in the Cayley graph of $F_{K,w}$ labelled one of the generators $q_1, \ldots, q_s$ (which have non-abelian centralizers). Let $x$ be one of the generators $u_1, \ldots, u_s$, and $f$ be an edge connecting $w=u^g$ with $w_1=x^{g_1}$, for some $g_1 \in \GG$. As in the proof of Lemma \ref{bs}, this means that $gC(u) \cap g_1 C(x) \neq \emptyset $, so $w_1=x^{g_1}=x^{gg_2}$, where $g_2 \in C(u)= \langle u, u_1, \ldots, u_k \rangle .$ Note that $g_2$ is defined up to multiplication on the right by an element of the centraliser of $x$. We can suppose that $g_2 \in \langle u_1, \ldots, u_k \rangle $, since $u$ commutes with $u_1, \ldots, u_k$, in particular with $x$ as well. Then $w_1=gg_2xg_2^{-1}g^{-1}=y{x^g}y^{-1}$, where $y=g_2^g$ is the word obtained from $g_2$ by replacing each $u_i$ by $u_i^g$, for $i=1, \ldots, k$, so $y$ can be thought of as an element of $F(u_1^g, \ldots, u_k^g)$. Now let $\overline{\pi}(f)$ be the line in the Cayley graph of $F_{K,w}$ passing through the vertex corresponding to $\pi(y)$, and with edges labeled by the element $\pi(x^g)=q$, which is equal to one of the elements $q_1, \ldots, q_s$.

		Note that $\overline{\pi}$ is well-defined. Indeed, suppose we can write $w_1$ in two ways:  $w_1=x^{g_1}=x^{gg_2}$ and  $w_1=x^{g_1'}=x^{gg_2'}$, where $g_2, g_2' \in \langle u_1, \ldots, u_k \rangle$, and $x$ is one of $u_1, \ldots, u_s$. Thus in $\GG$ we have $x^{gg_2}=x^{gg_2'}$, so $x^{g_2}=x^{g_2'}$. This can be considered as an equality in the free group on $u_1, \ldots, u_k$. Then $[x, g_2^{-1}g_2']=1$, so $g_2^{-1}g_2'=x^k$ for some $k \in \mathbb{Z}$, which means that $g_2'=g_2x^k$, and then $y'=y(x^g)^k$, where $y=g_2^g, \: y'=g_2'^g$, which means that $\pi(y')=\pi(y)q^k$, so the line in the Cayley graph is defined correctly.

		It follows from the definition of $\overline{\pi}$ that it is surjective. Now we show that $\overline{\pi}$ is injective. Suppose $f$ connects $w$ with $w_1=x^{g_1}$, anf $f'$ connects $w$ with $w_1'=x^{g_1'}$, $g_1=gg_2, \: y=g_2^g$ and $g_1'=gg_2', \: y'=g_2'^g$, where $g_2, g_2' \in \langle u_1, \ldots, u_k \rangle$.  If $\overline{\pi}(f)=\overline{\pi}(f')$, then  $\pi(y)=\pi(y')q^k$, so $y=y'(x^g)^k$ for some $k \in \mathbb{Z}$. Then $g_2=g_2'x^k$, and thus $g_1=g_1'x^k$, so $w_1=w_1'$. This shows that $\overline{\pi}$ is bijective.

		Note that $C_\GG(w)$ acts on the set of edges incident to the vertex $w$ of $\widetilde{\Delta}^e$, and thus $C_K(w)$ also acts on this set. The action of $w$ is trivial, so this induces the action of $C_K(w) / Z(C_K(w))$ on the set of edges incident to $w$. On the other hand, $P_{K,w}=\pi_C(C_K(w) / Z(C_K(w)))\le F_{K,w}$ acts on the Cayley graph of $F_{K,w}$ in the natural way (by left multiplication), and so $P_{K,w}$ also acts on the set of lines in the Cayley graph of $F_{K,w}$ defined above. Recall that $\pi_C$ induces an isomorphism $C_K(w) / Z(C_K(w)) \cong \pi_C(C_K(w) / Z(C_K(w)))$. It follows from the definition of $\overline{\pi}$ that for every edge $f$ of $\widetilde{\Delta}^e$ incident to $w$ and every $h \in C_K(w) / Z(C_K(w))$ we have $\overline{\pi}(h \cdot f)=\pi_C(h) \cdot \overline{\pi}(f)$, where we denote both actions defined above by dots. This means that the bijection $\overline{\pi}$ induces a bijection between the orbits of the edges incident to $w$ under the action of $C_K(w)$ on one side, and the orbits of the lines (as above) in the Cayley graph of $F_{K,w}$ under the action of $P_{K,w}$ on the other side.  By Lemma \ref{l3}, the orbits of the edges incident to $w$ under the action of $C_K(w)$ are in one-to-one correspondence with the edges of $\Psi(K)$ incident to $v=\gamma_K(w)$. And the orbits of the lines (as above) in the Cayley graph of $F_{K,w}$ under the action of $P_{K,w}$ are in a natural one-to-one correspondence with the cycles in the Schreier graph of $P_{K,w}$ in $F_{K,w}$, which are labelled by the generators $q_1, \ldots, q_s$. This defines the desired bijection $\pi_*$ and proves the first claim of the Lemma.

		Now we turn to the second claim. Let $f$ be an edge of $\widetilde{\Delta}^e$ incident to $w$ such that $\gamma_K(f)=p$, and let $f$ connect $w$ with $w_1$, so $\gamma_K(w_1)=v_1$. Recall that the label of $p$ at $v_1$ is equal to the label of $f$ at $w_1$, $\overline{l}_{w_1}(f)$, which is the minimal positive integer $k$ such that there exists an integer $l$  such that $w_1^k w^l \in K$. Note that this means that $\pi(w_1^k)=\pi(w_1^kw^l) \in P_{K,w}$. Since $Ker(\pi)= \langle w \rangle$, it follows that $\overline{l}_{w_1}(f)$ is equal to the minimal positive integer $k'$ such that $(\pi(w_1))^{k'} \in P_{K,w}$. Let $w=u^g$ and $w_1=x^{gg_2}$ as above, where $g_2 \in \langle u_1, \ldots, u_k \rangle$.
		
		Denote, as above, $\pi(x^g)=q, \: g_2^g=y$.
		We have $\pi(w_1)=\pi(x^{gg_2})= \pi(y)q\pi(y)^{-1},$ and $\overline{\pi}(f)$ is the line passing through $\pi(y)$ and with edges labelled by $q$. Thus $\overline{l}_{w_1}(f)$ is equal to the minimal positive integer $k'$ such that
		$$
		(\pi(w_1))^{k'}=(\pi(y)q\pi(y)^{-1})^{k'}=\pi(y)q^{k'}\pi(y)^{-1} \in P_{K,w},
		$$
		which is equivalent to $r\pi(y)=\pi(y)q^{k'}$ for some $r \in P_{K,w}$. Thus, $\overline{l}_{w_1}(f)$ is equal to the minimal positive exponent $k'$ of $q$ for which points $\pi(y)$ and $\pi(y)q^{k'}$ on the line $\overline{\pi}(f)$ in the Cayley graph of  $F_{K,w}$ are equivalent under the action of $P_{K,w}$, and this is exactly the length of the corresponding cycle $\pi_*(f)$ in the Schreier graph of $P_{K,w}$ in $F_{K,w}$. This proves the second claim.
		\end{proof}

		\begin{lemma}\label{l4-4}
		In the above notation, let $w, w'$ be two vertices of $\widetilde{\Delta}^e$ in the same orbit under the action of $K$. Then there is a natural isomorphism $\alpha$ between $F_{K,w}$ and $F_{K,w'}$, which restricts to an isomorphism $\alpha_P$ between $P_{K,w}$ and $P_{K,w'}$ and respects the bijections from {\rm Lemma \ref{l4}}.
		\end{lemma}
		\begin{proof}
		Let $w'=w^h$, where $h \in K$, and let $u, u_1, \ldots, u_k$ be as above. Then
		$$
		\begin{array}{l}
		w=u^g, \: C_\GG(w) \cong \langle u^g \rangle \times  F(u_1^g, \ldots, u_k^g),\\
		w'=u^{hg}, \: C_\GG(w')=(C_\GG(w))^h \cong \langle u^{hg} \rangle \times  F(u_1^{hg}, \ldots, u_k^{hg}).
		\end{array}
		$$
		Also
		$$
		\begin{array}{l}
		F_{K,w}=F(q_1, \ldots, q_k), \: F_{K,w'}=F(q_1', \ldots, q_k'),\\
		\pi:C_\GG(w) \rightarrow F_{K,w}, \:\pi(w)=1, \: \pi(u_i^g)=q_i, \: i=1, \ldots, k, \\
		\pi':C_\GG(w') \rightarrow F_{K,w'}, \:\pi(w^h)=1, \: \pi(u_i^{hg})=q_i', \: i=1, \ldots, k.
		\end{array}
		$$
		Let $\overline{\alpha}$ be the isomorphism $\overline{\alpha}: C_\GG(w) \rightarrow C_\GG(w'), \: \overline{\alpha}(t)=t^h$ for $t \in C_\GG(w)$. Define the isomorphism $\alpha: F_{K,w} \rightarrow F_{K,w'}, \: \alpha(q_i)=q_i', \: i=1, \ldots, k$. Then $\alpha\pi=\pi'\overline{\alpha}$ by definitions.
		
		We have
		$$
		C_K(w')=K \cap C_\GG(w')= K \cap (C_\GG(w))^h = (K \cap C_\GG(w))^h=C_K(w)^h,
		$$
		so $\overline{\alpha}(C_K(w))=C_K(w')$. Note that  $P_{K,w}=\pi(C_K(w)), \: P_{K,w'}=\pi'(C_K(w'))$, so
		$$
		\alpha(P_{K,w})= \alpha\pi(C_K(w))=\pi'\overline{\alpha}(C_K(w))=\pi'(C_K(w'))=P_{K,w'},
		$$ so $\alpha$ induces an isomorphism $\alpha_P$ between $P_{K,w}$ and $P_{K,w'}$. Thus $\alpha$ also induces an isomorphism between the Schreier graphs of $P_{K,w}$ in $F_{K,w}$ and of $P_{K,w'}$ in $F_{K,w'}$ in a natural way.
		
		Recall that $\pi_*$ is a bijection between the edges of $\Psi(K)$ incident to $v=\gamma_K(w)$ and those cycles in the Schreier graph of $P_{K,w}$ in $F_{K,w}$ which correspond to the generators $q_1, \ldots, q_s$. Analogously $\pi'_*$ is a bijection between the edges of $\Psi(K)$ incident to $v=\gamma_K(w')=\gamma_K(w)$ and those cycles in the Schreier graph of $P_{K,w'}$ in $F_{K,w'}$ which correspond to the generators $q_1', \ldots, q_s'$. It follows from the definition of $\pi_*, \pi'_*$ that for every edge $p$  of $\Psi(K)$ incident to $v$ the cycle $\pi_*(p)$ goes into the cycle $\pi'_*(p)$ under the isomorphism induced by $\alpha$, i.e., $\alpha$ respects the bijections.
	\end{proof}

	If $v$ is some vertex of $\Psi(H)$, then by Lemma \ref{l4-4} different choices of $w$ in $\widetilde{\Delta}^e$ such that $\gamma_K(w)=v$ give the same groups $F_{K,w}$, subgroups $P_{K,w}$, as well as bijections $\pi_*$, as in Lemma \ref{l4}, up to isomorphism. This means that the exact choice of such $w$ is unimportant, and, abusing the notation, we will denote $F_{K,v}=F_{K,w}, \: P_{K,v}=P_{K,w}$, where the vertex $w$ is any vertex of $\widetilde{\Delta}^e$ such that $\gamma_K(w)=v$.

	\begin{lemma}\label{lll}
		In the above notation, with $v$ being a vertex of $\Psi(K)$ and $u=\delta_K(v)$, let $u_0$ be some vertex of $\Delta$ of degree more than 1, connected to $u$ by an edge $e_0$. If $f_1, \ldots, f_k$ are all edges of $\Psi(K)$ incident to the vertex $v$ such that $\delta_K(f_i)=e_0$ for all $i=1,\ldots,k$, and each $f_i$ connects $v$ with the vertex $v_i$ {\rm(}$i=1,\ldots, k${\rm)}, then
		\begin{equation}\label{sum}
		\sum_{i=1}^k l_{v_i}(f_i)= |F_{K,v}:P_{K,v}|.
		\end{equation}	
		If $z_1, \ldots, z_p$ are all edges of $\Psi(K)$ incident to $v$, and each $z_i$ connects $v$ with the vertex $w_i$, then
		\begin{equation}\label{sum1}
		\sum_{i=1}^p l_{w_i}(z_i)= r \: |F_{K,v}:P_{K,v}|.
		\end{equation}	
		where $r$ is the number of vertices of degree greater than 1 adjacent to $u$ in $\Delta$.
	\end{lemma}
	\begin{proof}
		According to Lemma \ref{l4}, there is a bijection between the set of edges $f_1, \ldots, f_k$ and the set of cycles in the Schreier graph of $P_{K,v}$ in $F_{K,v}$ which correspond to conjugates of $u_0$, and the length of each cycle is equal to the corresponding edge label. The sum of lengths of these cycles is equal to the number of vertices in the Schreier graph, which is equal to $|F_{K,v}:P_{K,v}|$, thus (\ref{sum}) follows. Summing all the equalities of the form (\ref{sum}) for all vertices of $\Delta$ of degree greater than 1 adjacent to $u$ gives (\ref{sum1}).
	\end{proof}
	
	As above, suppose now $\Delta$ and $\Delta'$ are finite trees such that $\GG=\mathbb{G}(\Delta)$ and $\GG'=\mathbb{G}(\Delta')$ are commensurable, and so there exist finite index subgroups $H \le \GG$ and $H' \le \GG'$ and an isomorphism $\varphi: H \rightarrow H'$, and $\Psi(H)$, $\Psi(H')$ are defined as above. Recall that for a vertex $v$ of $\Psi(H)$ and a vertex $v'$ of $\Psi(H')$ the finite index subgroups $P_{H,v} \le F_{H,v}$ and $P_{H',v'} \le F_{H',v'}$ are defined as in
 {\rm Lemma \ref{l4-1}}.
 	
	\begin{lemma}\label{equations1}
		Let $v$ be a vertex of $\Psi(H)$, and let $v'=\varphi_*(v)$ be the corresponding vertex of $\Psi(H')$, where $\varphi_*$ is the graph isomorphism between $\Psi(H)$ and $\Psi(H')$ from {\rm Lemma \ref{l2}}. Let $\delta_H(v)=u$ and $\delta_{H'}(v')=u'$, and suppose $u$ has degree $t$ in $\Delta$, and $u'$ has degree $t'$ in $\Delta'$. Then 
		\begin{equation}\label{index}
		(t - 1) \: |F_{H,v}:P_{H,v}| = (t' - 1) \: |F_{H',v'}:P_{H',v'}|.
		\end{equation}
	\end{lemma}	
	\begin{proof}
		Let $w \in \widetilde{\Delta}^e$ be such that $\gamma_H(w)=v$, and think of $w$ as an element of $\GG$. Let $l$ be such positive integer that $w^l \in H$. Then
		$$
		C_H(w)=C_H(w^l) \cong C_{H'}(\varphi(w^l))=C_{H'}(\varphi(w)).
		$$
		Note that $\gamma_{H'}(\varphi(w))=v'$. Thus,
		$$
		P_{H,v} \cong C_H(w) / Z(C_H(w)) \cong  C_{H'}(\varphi(w)) / Z(C_{H'}(\varphi(w))) \cong P_{H',v'}.
		$$
		In particular, $rk(P_{H,v})=rk(P_{H',v'})$. Then, by Schreier index formula,
		$$
		|F_{H,v}:P_{H,v}| (rk(F_{H,v}) - 1)=rk(P_{H,v}) - 1=rk(P_{H',v'}) - 1=  |F_{H',v'}:P_{H',v'}| (rk(F_{H',v'}) - 1).
		$$
		This implies (\ref{index}).
	\end{proof}

	Combining the two previous lemmas, we establish a linear relation between the minimal quotient exponents that we record in the following corollary.
	
	\begin{corollary} \label{cor:relminqexp}
		Let $v$ be a vertex of $\Psi(H)$, and let $v'=\varphi_*(v)$. Let $\delta_H(v)=u$ and $\delta_{H'}(v')=u'$, and suppose $u$ has degree $t$ in $\Delta$, and $u'$ has degree $t'$ in $\Delta'$.
		Let $e_1, \ldots, e_k$ be all edges of $\Psi(H)$ incident to $v$, where each $e_i$ connects $v$ with the vertex $w_i$, and let $e'_i = \varphi_*(e_i)$, $w'_i = \varphi_*(w_i)$, $i=1, \ldots, k$. Let also $r$ be the number of non-leaf vertices adjacent to $u$ in $\Delta$, and $r'$ be the number of non-leaf vertices adjacent to $u'$ in $\Delta'$.
		Then
		$$
		\frac{t - 1}{r} \: \sum_{i=1}^k l_{w_i}(e_i) = \frac{t' - 1}{r'} \: \sum_{i=1}^k l_{w'_i}(e'_i).
		$$
	\end{corollary}
	\begin{proof}
		Since $\varphi_*$ is a graph isomorphism, see Lemma \ref{l2}, $e'_1, \ldots, e'_k$ are all edges of $\Psi(H')$ incident to $v'$. By Lemma \ref{lll} we have
		$$
		\sum_{i=1}^k l_{w_i}(e_i) = r|F_{H,v}:P_{H,v}|, \quad \sum_{i=1}^k l_{w_i'}(e'_i)= r'|F_{H',v'}:P_{H',v'}|.
		$$
		The statement now follows from Lemma \ref{equations1}.
	\end{proof}
	
	\subsubsection{Linear relations between minimal exponents} \label{sec:minexp}
		
We recall the definitions of labels of vertices and edges, see Definition \ref{defn:minexps}.

Let $f$ be an edge of $\widetilde{\Delta}^e$ incident to a vertex $w$, and $f'=\overline{\varphi}(f), \: w'=\overline{\varphi}(w)$, where $\overline{\varphi}$ is as in Lemma \ref{l2}. By $\overline{L}(w)$, $\overline{L'}(w)$ we denote the labels of the vertices $w$ in $\widetilde{\Delta}^e$, $w'$ in $\widetilde{\Delta}'^e$ respectively. By $\overline{l}_w(f)$ we denote the label of the edge $f$ with respect to $w$ in $\widetilde{\Delta}^e$, and by $\overline{l'}_w(f)$ we denote the label of the edge $f'$ with respect to $w'$ in $\widetilde{\Delta}'^e$.

Analogously, let $p$ be an edge of $\Psi(H)$ incident to a vertex $v$, and $p'=\varphi_*(p), \: v'=\varphi_*(v)$. By $L(v)$, $L'(v)$ we denote the labels of the vertices $v$ in $\Psi(H)$, $v'$ in $\Psi(H')$ respectively. By $l_v(p)$ we denote the label of the edge $p$ with respect to $v$ in $\Psi(H)$, and by $l'_v(p)$ we denote the label of the edge $p'$ with respect to $v'$ in $\Psi(H')$.

\begin{lemma}\label{prop1}
	Let $p$ be an edge of $\Psi(H)$ connecting vertices $v_1$ and $v_2$. Then
	\begin{equation}\label{prop}	
	\frac{L(v_1)}{l_{v_1}(p)}=\frac{L(v_2)}{l_{v_2}(p)}=r,
	\end{equation}	
	where $r$ is some positive integer.
\end{lemma}
\begin{proof}
	Let $f$ be some edge of $\widetilde{\Delta}^e$ such that $\gamma_H(f)=p$, and $w_1,w_2$ be the ends of $f$, so that $\gamma_H(w_1)=v_1, \: \gamma_H(w_2)=v_2$. Then we can rewrite (\ref{prop}) as
	\begin{equation}\label{lift}
	\frac{\overline{L}(w_1)}{\overline{l}_{w_1}(f)}=\frac{\overline{L}(w_2)}{\overline{l}_{w_2}(f)}=r.
	\end{equation}
	By definition $\overline{l}_{w_1}(f)$ is the minimal positive integer $k$ such that there exists an integer $l$  such that $w_1^k w_2^l \in H$. Note that if $k'$ is some other positive integer such that there exists an integer $l'$ such that $w_1^{k'} w_2^{l'} \in H$, then $\overline{l}_{w_1}(f)|k'$. Indeed, denote $k_0=\overline{l}_{w_1}(f)$ and let $d$ be the greatest common divisor of $k_0$ and $k'$. Then there exist integers $\alpha, \beta$ such that $\alpha k_0 + \beta k' = d$. We have $w_1^{k_0}w_2^l \in H$ and $w_1^{k'}w_2^{l'} \in H$, so, since $w_1$ and $w_2$ commute,
	$$
	(w_1^{k_0}w_2^l)^{\alpha}(w_1^{k'}w_2^{l'})^{\beta} = w_1^{\alpha k_0 + \beta k'}w_2^{\alpha l + \beta l'} =w_1^d w_2^{\alpha l + \beta l'} \in H,
	$$
	so $d \geq k_0$, but since $d=gcd(k_0,k')$, this means that $d=k_0$ and thus $k_0=\overline{l}_{w_1}(f)|k'$.
	
	Recall that $\overline{L}(w_1)$ is the minimal positive integer $k$ such that $w_1^k \in H$, so $w_1^{\overline{L}(w_1)} \in H$ and it follows from above that $\overline{l}_{w_1}(f)|\overline{L}(w_1)$. Thus, $\frac{\overline{L}(w_1)}{\overline{l}_{w_1}(f)}=r_1$ is a positive integer. Analogously $\frac{\overline{L}(w_2)}{\overline{l}_{w_2}(f)}=r_2$ is a positive integer.
	
	It remains to show that $r_1=r_2$. Indeed, suppose that $l$ is such that $w_1^{\overline{l}_{w_1}(f)}w_2^l \in H$.
	As shown above, this means that $\overline{l}_{w_2}(f)|l$, so let $l=\overline{l}_{w_2}(f) q$.
	We have
	$$
	(w_1^{\overline{l}_{w_1}(f)}w_2^l)^{r_2}=w_1^{\overline{l}_{w_1}(f)r_2} w_2^{lr_2} = w_1^{\overline{l}_{w_1}(f)r_2} w_2^{\overline{l}_{w_2}(f) qr_2} =  w_1^{\overline{l}_{w_1}(f)r_2} w_2^{\overline{L}(w_2)q} \in H.
	$$
	But $w_2^{\overline{L}(w_2)} \in H$, so $w_2^{\overline{L}(w_2)q} \in H$. Thus  $w_1^{\overline{l}_{w_1}(f)r_2} \in H$. This means that $\overline{l}_{w_1}(f)r_2 \geq \overline{L}(w_1)=\overline{l}_{w_1}(f) r_1$, so $r_2 \geq r_1$. Analogous argument shows that $r_1 \geq r_2$. Thus $r_1=r_2$.
\end{proof}

	\begin{lemma}\label{proportions}
		Let $p$ be an edge of $\Psi(H)$ connecting vertices $v_1$ and $v_2$. Then
		\begin{equation}\label{eq3}
		\frac{L(v_1)}{l_{v_1}(p)} = \frac{L'(v_1)}{l'_{v_1}(p)} = \frac{L(v_2)}{l_{v_2}(p)} = \frac{L'(v_2)}{l'_{v_2}(p)}=q,
		\end{equation}
		where $q$ is some positive integer.
	\end{lemma}
	\begin{proof}
		We already know from Lemma \ref{prop1} that
		$$
		\frac{L(v_1)}{l_{v_1}(p)} = \frac{L(v_2)}{l_{v_2}(p)}=q_1, \quad  \frac{L'(v_1)}{l'_{v_1}(p)} = \frac{L'(v_2)}{l'_{v_2}(p)}=q_2,
		$$ where $q_1$ and $q_2$ are some positive integers. Thus it suffices to show that $q_1=q_2$, or
		\begin{equation}\label{nnn}
		\frac{L(v_1)}{l_{v_1}(p)} = \frac{L'(v_1)}{l'_{v_1}(p)}.
		\end{equation}
		Let $f$ be some edge in $\widetilde{\Delta}^e$ such that $\gamma_H(f)=p$, and $w_1,w_2$ be the ends of $f$, so that $\gamma_H(w_1)=v_1, \: \gamma_H(w_2)=v_2$. Let also $w_1'=\overline{\varphi}(w_1), \: w_2'=\overline{\varphi}(w_2)$. Then we can rewrite (\ref{nnn}) as
		\begin{equation}\label{lift2}
		\frac{\overline{L}(w_1)}{\overline{l}_{w_1}(f)}=\frac{\overline{L'}(w_1)}{\overline{l'}_{w_1}(f)}.
		\end{equation}
		
		Recall that $\overline{l}_{w_1}(f)$	is the minimal positive integer $k$ such that there exists an integer $l$  such that $w_1^k w_2^l \in H$. Fix such number $l$. Denote $\overline{L}(w_1)=m_1$, $\overline{L}(w_2)=m_2$, $\overline{L'}(w_1)=m'_1$, $\overline{L'}(w_2)=m'_2$. Then $w_1^{\overline{l}_{w_1}(f)} w_2^l \in H$, and the element $y=(w_1^{\overline{l}_{w_1}(f)} w_2^l)^{m_1m_2}$ is a $m_1m_2$ power in $H$. Note that
		$$
		y = (w_1^ {m_1})^{\overline{l}_{w_1}(f)m_2} \: (w_2^{m_2})^{lm_1}.
		$$
		Note also that $w_1^{m_1}, w_2^{m_2} \in H$ by definition of labels, and $\varphi(w_1^{m_1})=(w_1')^{m'_1}$, $\varphi(w_2^{m_2})=(w_2')^{m'_2},$ as shown in the proof of Lemma \ref{l2}. Thus
		$$
		\varphi(y)=((w_1')^ {m_1'})^{\overline{l}_{w_1}(f)m_2} \: ((w_2')^{m_2'})^ {lm_1}
		$$
		is a $m_1m_2$ power in $H'$: $\varphi(y)=z^{m_1m_2}$ for some $z \in H'$. But roots in RAAGs are unique, see Section \ref{sec:prelim}, so
		$$
		z = (w_1')^ {m_1'\overline{l}_{w_1}(f)/m_1} \: (w_2')^{m_2' l/m_2} \in H'.
		$$
		By definition of $\overline{l'}_{w_1}(f)$, this means that $m_1'\overline{l}_{w_1}(f)/m_1 \geq 	\overline{l'}_{w_1}(f)$, so
		$$ \frac{\overline{L'}(w_1)}{\overline{l'}_{w_1}(f)} \geq \frac{\overline{L}(w_1)}{\overline{l}_{w_1}(f)}.
		$$
		The same argument with the roles of $\GG$ and $\GG'$ interchanged shows that
		$$
		\frac{\overline{L}(w_1)}{\overline{l}_{w_1}(f)} \geq \frac{\overline{L'}(w_1)}{\overline{l'}_{w_1}(f)}.
		$$
		Thus
		$$
		\frac{\overline{L}(w_1)}{\overline{l}_{w_1}(f)} = \frac{\overline{L'}(w_1)}{\overline{l'}_{w_1}(f)}.
		$$
		This proves the Lemma.
	\end{proof}
	
	The above lemmas imply that the labels of edges satisfy the following equations.
	\begin{lemma}\label{together}
		Suppose $v$ is some vertex of $\Psi(H)$, and $v'=\varphi_*(v)$, where $\varphi_*$ is the graph isomorphism between $\Psi(H)$ and $\Psi(H')$ from {\rm Lemma \ref{l2}}. Suppose $\delta_H(v)$ has degree $t$ in $\Delta$ and is adjacent to $r$ vertices of degree more than 1 in $\Delta$, and $\delta_{H'}(v')$ has degree $t'$ in $\Delta'$, and is adjacent to $r'$ vertices of degree more than 1 in $\Delta'$.
		Let $f_1, \ldots, f_k$ be all edges of $\Psi(H)$ adjacent to $v$, and suppose $\alpha_i=l_v(f_i)$, $\alpha'_i=l'_v(f_i)$, $i=1, \ldots, k$, are their labels.
		Then we have
		\begin{equation}\label{eq1}
		\frac{t-1}{r} \sum_{i=1}^k \alpha_i L(v_i) = \frac{t'-1}{r'} \sum_{i=1}^k \alpha_i L'(v_i),
		\end{equation}
		\begin{equation}\label{eq2}
		\frac{t-1}{r} \sum_{i=1}^k \alpha'_i L(v_i) = \frac{t'-1}{r'} \sum_{i=1}^k \alpha'_i L'(v_i).
		\end{equation}
	\end{lemma}	
	
	\begin{proof}
		Let the edge $f_i$ connect $v$ to $v_i$, $i=1,\ldots,k$. By Corollary \ref{cor:relminqexp}, we have that
		$$
		\frac{t - 1}{r} \: \sum_{i=1}^k l_{v_i}(f_i) = \frac{t' - 1}{r'} \: \sum_{i=1}^k l'_{v_i}(f_i).
		$$
		 Multiplying both sides by $L(v)$, we obtain
		\begin{equation}\label{qw}
		\frac{t - 1}{r} \: \sum_{i=1}^k l_{v_i}(f_i)L(v) = \frac{t' - 1}{r'} \: \sum_{i=1}^k l'_{v_i}(f_i)L(v).
		\end{equation}
		 Multiplying both sides by $L'(v)$, we obtain
		\begin{equation}\label{qw'}
		\frac{t - 1}{r} \: \sum_{i=1}^k l_{v_i}(f_i)L'(v) = \frac{t' - 1}{r'} \: \sum_{i=1}^k l'_{v_i}(f_i)L'(v).
		\end{equation}
		According to Lemma \ref{proportions}, we have
		$$
		\frac{L(v)}{l_{v}(f_i)} = \frac{L'(v)}{l'_v(f_i)} = \frac{L(v_i)}{l_{v_i}(f_i)} = \frac{L'(v_i)}{l'_{v_i}(f_i)}
		$$
		for all $i=1,\ldots,k$, so
		$$
		l_{v_i}(f_i)L(v)= L(v_i)l_v(f_i), \quad l'_{v_i}(f_i)L(v)=L'(v_i)l_v(f_i),
		$$
		$$
		l_{v_i}(f_i)L'(v)= L(v_i)l'_v(f_i), \quad l'_{v_i}(f_i)L'(v)=L'(v_i)l'_v(f_i).
		$$
		Thus (\ref{qw}), (\ref{qw'}) imply (\ref{eq1}), (\ref{eq2}).
	\end{proof}

\begin{defn}[Type of a vertex]\label{type} If $v$ is a vertex of $\Psi(H)$ such that $\delta_H(v)=u_1$, and $\delta_{H'}(\varphi_*(v))=u_2$, then we say that $v$ {\it is a vertex of type $u_1/u_2$}, or $v$ {\it is a $u_1/u_2$-type vertex}. Note that a vertex of $\Psi(H)$ of type $u_1/u_2$ can only be adjacent to a vertex of $\Psi(H)$ of type $v_1/v_2$ if $u_1$ and $v_1$ are adjacent in $\Delta$ and $u_2$ and $v_2$ are adjacent in $\Delta'$.
\end{defn}
	
\subsection{Infinitely many commensurability classes of RAAGs defined by trees of diameter 4} \label{Sec:InfCom}

	We now turn our attention to the proofs of the main results. We first show that there are infinitely many different commensurability classes inside the class of RAAGs defined by trees of diameter 4. This is a particular case of Theorem \ref{th3} , but we give a separate proof for two reasons: the first one, it is technically easier and it helps as a warm up for the general case; secondly, for the reader interested only in the qualitative aspect of the result, it suffices to read our Theorem \ref{th2}.
	
	For $0<d_1 \leq d_2$ we denote by $P_{d_1,d_2}$ the tree of diameter 4 with 2 pivots, one of degree $d_1$ and the other of degree $d_2$, and no hair vertices, see Figure \ref{fig:tree2}. Thus $P_{d_1,d_2}=T((d_1,1),(d_2,1);0)$ if $d_1 < d_2$ and $P_{d_1,d_1}=T((d_1,2);0)$.
	
		\begin{figure}[!h]
		\centering
		\includegraphics[keepaspectratio,width=3.2in]{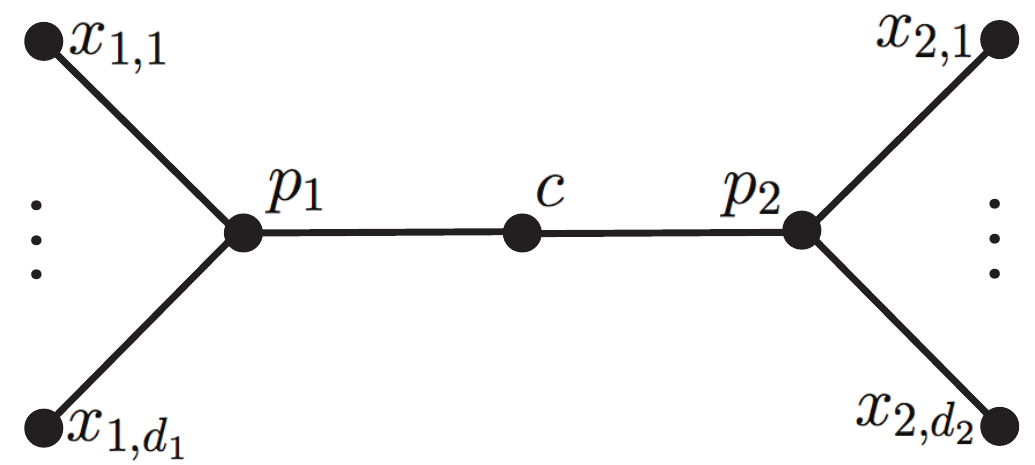}
		\caption{Tree of diameter 4 with 2 pivots} \label{fig:tree2}
	\end{figure}

	\begin{theorem}[Infinitely many commensurability classes]\label{th2}
		Let $0<m_1 \leq m_2$ and $0<n_1 \leq n_2$. Suppose the groups $\GG=\mathbb{G}(P_{m_1,m_2})$ and $\GG'=\mathbb{G}(P_{n_1,n_2})$ are commensurable. Then $\frac{m_1}{n_1}=\frac{m_2}{n_2}$.
		
		In particular, there are infinitely many commensurability classes among RAAGs defined by trees of diameter 4.
	\end{theorem}
	In fact, we will see in the next section that $\GG=\mathbb{G}(P_{m_1,m_2})$ and $\GG'=\mathbb{G}(P_{n_1,n_2})$ are commensurable if and only if $\frac{m_1}{n_1}=\frac{m_2}{n_2}$.
	\begin{proof}[Proof of Theorem \ref{th2}]
		Let $c$ be the center vertex of $P_{m_1,m_2}$ and $c'$ be the center vertex of $P_{n_1,n_2}$. Let $p_1,p_2$ be the pivots of $P_{m_1,m_2}$ of degrees $m_1+1$, $m_2+1$ correspondingly, and $p_1',p_2'$ be the pivots of $P_{n_1,n_2}$ of degrees $n_1+1,n_2+1$ correspondingly.

		Let $H \leq \GG$ and $H' \leq \GG'$ be finite index subgroups such that $H \cong H'$.
		From Lemma \ref{together} and Lemma \ref{lll} applied to both $\Delta=P_{m_1,m_2}$ and $\Delta'=P_{n_1,n_2}$ we have a system of equations on the labels of vertices and edges of $\Psi(H)$. We will show that this system of equations does not have any positive integer solutions unless $\frac{m_1}{n_1}=\frac{m_2}{n_2}$.

		Note that there could exist vertices of the following nine types in $\Psi(H)$:
		 $p_1/p_1', \: p_1/p_2', \: p_2/p_1', \: p_2/p_2',$ $c/p_1', \: c/p_2', \: p_1/c', \: p_2/c', \: c/c'$, see Definition \ref{type}. It also follows from Definition \ref{type} that the following statements hold:
		\begin{itemize}
			\item Every vertex of type $p_1/p_1', \: p_1/p_2', \: p_2/p_1', \: p_2/p_2'$ can be connected only with $c/c'$-type vertices;
			\item Every $c/c'$-type vertex can be connected only with vertices of type $p_1/p_1', \: p_1/p_2', \: p_2/p_1', \: p_2/p_2'$;
			\item Every $p_1/c'$-type and $p_2/c'$-type vertex can be connected only with $c/p_1'$-type and $c/p_2'$-type vertices;
			\item Every $c/p_1'$-type and $c/p_2'$-type vertex can be connected only with $p_1/c'$-type and $p_2/c'$-type vertices.	
		\end{itemize}
		Thus, there are two cases: either $\Psi(H)$ has only vertices of type $p_1/p_1', \: p_1/p_2', \: p_2/p_1', \: p_2/p_2'$ and $c/c'$, or $\Psi(H)$ has only vertices of type $p_1/c'$, $p_2/c'$, $c/p_1'$ and $c/p_2'$.
		
		{\it Case 1.} Suppose first $\Psi(H)$ has only vertices of type $p_1/p_1', \: p_1/p_2', \: p_2/p_1', \: p_2/p_2'$ and $c/c'$.

		Let $v$ be a $p_i/p_j'$-type vertex, where $i=1,2$, $j=1,2$.
		Then all vertices adjacent to $v$ are of type $c/c'$. Let $f_1, \ldots, f_S$ be all edges incident to $v$, and let $f_s$ connect $v$ with $v_s$ for $s=1,\ldots,S$ (possibly some of $v_i$ coincide). 	Then, by Lemma \ref{together}, in the notations of which we have $t=m_i+1, \: r=1, \: t'=n_j+1, \: r'=1$, we have
		\begin{equation}\label{s}
		m_i \: \sum_{s=1}^S L(v_s)l_v(f_s) = n_j \: \sum_{s=1}^S L'(v_s)l_v(f_s).
		\end{equation}
		Let $w_1, \ldots, w_Q$  be all the $c/c'$-type vertices of $\Psi(H)$. For every vertex $w_q, \: q=1,\ldots, Q$,  consider all the edges incident to $w_q$ which finish in $p_i/p_j'$-type vertices (for fixed $i=1,2$ and $j=1,2$), and let $X^q_{ij}$ be the sum of the $l$-labels of these edges at the $p_i/p_j'$-type ends; we let $X^q_{ij}=0$ if there are no such edges.
		It follows from Lemma \ref{lll} that for all $q=1,\ldots, Q$ we have
		\begin{equation}\label{rrr}
		X_{11}^q + X_{12}^q = |F_{H,w_j}: P_{H,w_j}| = X_{21}^q + X_{22}^q. 		
		\end{equation}
		Now sum up the equalities of the form (\ref{s}) for all $p_i/p_j'$-type vertices $v$ (for fixed $i=1,2$ and $j=1,2$), and group the summands with the vertex label corresponding to the same $c/c'$-type vertex together.  We obtain
		\begin{equation}\label{bb'}
		m_1 \: \sum_{q=1}^Q L(w_q)X^q_{11}= n_1 \: \sum_{q=1}^Q L'(w_q)X^q_{11},
		\end{equation}
		\begin{equation}\label{db'}
		m_1 \: \sum_{q=1}^Q L(w_q)X^q_{12}= n_2 \: \sum_{q=1}^Q L'(w_q)X^q_{12},
		\end{equation}
		\begin{equation}\label{dd'}
		m_2 \: \sum_{q=1}^Q L(w_q)X^q_{21}= n_1 \: \sum_{q=1}^Q L'(w_q)X^q_{21},
		\end{equation}
		\begin{equation}\label{bd'}
		m_2 \: \sum_{q=1}^Q L(w_q)X^q_{22}= n_2 \: \sum_{q=1}^Q L'(w_q)X^q_{22}.
		\end{equation}
		
		Summing (\ref{bb'}) and (\ref{db'}) and using that $n_1 \leq n_2$, we get
		\begin{equation}\label{fff1}
		m_1 \: \sum_{q=1}^Q L(w_j)(X^q_{11}+X^q_{12})=n_1 \: \sum_{q=1}^Q L'(w_q)X^q_{11} + n_2 \:
		\sum_{q=1}^Q L'(w_q)X^q_{12} \geq n_1 \: \sum_{q=1}^Q L'(w_q)(X^q_{11}+X^q_{12}).
		\end{equation}
		Analogously, summing (\ref{bd'}) and (\ref{dd'}) and using that $n_1 \leq n_2$, we get
		\begin{equation}\label{kkk1}
		m_2 \: \sum_{q=1}^Q L(w_q)(X^q_{21}+X^q_{22}) = n_1 \: \sum_{q=1}^Q L'(w_q)X^q_{21} + n_2 \:
		\sum_{q=1}^Q L'(w_q)X^q_{22} \leq n_2 \:  \sum_{q=1}^Q L'(w_q)(X^q_{21} + X^q_{22}).
		\end{equation}
		Multiplying (\ref{fff1}) by $m_2$ and (\ref{kkk1}) by $m_1$ and using (\ref{rrr}), we obtain
		\begin{gather} \notag
		\begin{split}
		m_2 n_1 \: \sum_{q=1}^Q L'(w_q)(X^q_{11}+X^q_{12}) &\leq m_2 m_1 \: \sum_{q=1}^Q L(w_q)(X^q_{11}+X^q_{12})\\
		 & = m_1 m_2 \: \sum_{q=1}^Q L(w_q)(X^q_{21}+X^q_{22}) \leq m_1 n_2 \sum_{q=1}^Q L'(w_q)(X^q_{21} + X^q_{22})\\
		 &= m_1 n_2 \: \sum_{q=1}^Q L'(w_q)(X^q_{11}+X^q_{12}).
		\end{split}
		\end{gather}
		
		Since $L'(w_q)>0$ by definition and $X^q_{11}+X^q_{12} > 0$ by (\ref{rrr}) for every $q=1,\ldots,Q$, it follows that $m_2n_1\leq m_1n_2$, or $\frac{m_2}{m_1} \leq \frac{n_2}{n_1}$.
		
		Note that we can apply the same argument with the roles of $\GG$ and $\GG'$ interchanged, and thus get  $\frac{n_2}{n_1} \leq \frac{m_2}{m_1}$. This means that  $\frac{m_2}{m_1} = \frac{n_2}{n_1}$, so $\frac{m_1}{n_1}=\frac{m_2}{n_2}$, as desired.
		
		{\it Case 2.} Now suppose that $\Psi(H)$ has only vertices of type $p_1/c'$, $p_2/c'$, $c/p_1'$ and $c/p_2'$.
		
		Let $v$ be a $c/p_j'$-type vertex of $\Psi(H)$, where $j=1$ or $j=2$. Then all vertices adjacent to $v$ are of type $p_1/c'$ or $p_2/c'$. Let $f_1, \ldots,f_S$ be all edges incident to $v$, and let $f_s$ connect $v$ with $v_s$ for $s=1,\ldots,S$ (possibly some of the vertices $v_s$ coincide).
		Then, by Lemma \ref{together}, in the notations of which we have $t = 2, \: t' = n_j +1, \: r = 2, \: r' = 1$, we obtain
		\begin{equation}\label{s2}
		\: \sum_{s=1}^S L(v_s)l'_v(f_s) = 2n_j \: \sum_{s=1}^S L'(v_s)l'_v(f_s).
		\end{equation}

		Let $w_1, \ldots, w_Q$ be all the $p_1/c'$-type and $p_2/c'$-type vertices of $\Psi(H)$.  For every vertex $w_q, \: q=1,\ldots,Q$, consider all the edges incident to $w_q$ which finish in $c/p_1'$-type vertices (note that such edges always exist), and let $Z^q_1>0$ be the sum of the $l'$-labels of these edges at the $c/p_1'$-type ends. Analogously, consider all the edges incident to $w_q$ which finish in $c/p_2'$-type vertices (note that such edges should always exist), and let $Z^q_2>0$ be the sum of the $l'$-labels of these edges at the $c/p_2'$-type ends.
		Let $w_q'=\varphi_*(w_q)$ for  $q=1,\ldots, Q$.
		Then it follows from Lemma \ref{lll} that for every $q=1,\ldots,Q$ we have
		\begin{equation}\label{kkk}	
		|F_{H', w'_q}: P_{H', w'_q}|=Z^q_1=Z^q_2.
		\end{equation}
		Now sum up the equalities of the form (\ref{s2}) for all $c/p_1'$-type vertices $v$, and group the summands with the vertex label corresponding to the same $p_i/c'$-type vertex together.  We obtain
		\begin{equation}\label{rr10}
		\sum_{q=1}^Q L(w_q)Z^q_1= 2n_1 \: \sum_{q=1}^Q L'(w_q)Z^q_1.
		\end{equation}
		Analogously, summing the equalities of the form (\ref{s2}) for all $c/p_2'$-type vertices $v$ we get
		\begin{equation}\label{rr20}
		\sum_{q=1}^Q L(w_q)Z^q_2=  2n_2 \: \sum_{q=1}^Q L'(w_q)Z^q_2.
		\end{equation}
		Together with (\ref{kkk}), equalities (\ref{rr10}) and (\ref{rr20}) imply that $n_1=n_2$.
		
		Note that we can apply the same argument with the roles of $\GG$ and $\GG'$ interchanged, and thus get $m_1=m_2$. It follows that, in particular, $\frac{m_1}{n_1}=\frac{m_2}{n_2}$.

		This proves Theorem \ref{th2}.
	\end{proof}
	
\subsection{Criterion for non-commensurability of RAAGs defined by trees of diameter 4}\label{sec:noncomm}
	
	We now prove the main non-commensurability result of this paper.
	
	\begin{theorem}\label{th3}
		Let $T=T((m_1,1),\ldots, (m_k,1); 0)$ and $T'=T((n_1,1),\ldots, (n_l,1); 0)$ for some $k,l \geq 2$.
		Suppose $\GG=\GG(T)$ and $\GG'=\GG(T')$ are commensurable. Then $M(T)$ and $M(T')$ are commensurable, i.e., $k=l$, and 		
		\begin{equation}\label{mn}
		\frac{n_1}{m_1}=\frac{n_2}{m_2}= \ldots =\frac{n_k}{m_k}.
		\end{equation} 	
		Moreover, if $S=T((m, 2);0)$ for some $m \geq 2$, then $\GG(S)$ and $\GG(T')$ are not commensurable.
	\end{theorem}
	\begin{proof}[Proof of Theorem \ref{th3}]
		We first prove the first claim. Let $c$ be the central vertex of $T$ and $c'$ be the central vertex of $T'$. Let $p_1=p_{1,1},\ldots,p_k=p_{k,1}$ be the pivots of $T$, and $p_1'=p_{1,1}',\ldots,p_l'=p_{l,1}'$ be the pivots of $T'$.

		Let $H \leq \GG$ and $H' \leq \GG'$ be finite index subgroups such that $H \cong H'$.
		From Lemma \ref{together} and Lemma \ref{lll} applied to both $\Delta=T$ and $\Delta'=T'$ we have a system of equations on the labels of vertices and edges of $\Psi(H)$. We will show that this system of equations does not have any solutions in the positive integers unless $M(T)$ and $M(T')$ are commensurable.

		Below we always mean that $i$ ranges between $1$ and $k$, and $j$ ranges between $1$ and $l$, unless stated otherwise.
		
		Note that there could exist vertices of the following types in $\Psi(H)$: $c/c', \: p_i/p_j', \: p_i/c', \: c/p_j'$, see Definition \ref{type}. It also follows from Definition \ref{type} that the following statements hold:
		\begin{itemize}
			\item Every $p_i/p_j'$-type vertex can be connected only with $c/c'$-type vertices;
			\item Every $c/c'$-type vertex can be connected only with $p_i/p_j'$-type vertices;
			\item Every $p_i/c'$-type vertex can be connected only with $c/p_j'$-type vertices;
			\item Every $c/p_j'$-type vertex can be connected only with $p_i/c'$-type vertices.	
		\end{itemize}
		Thus, there are two cases: either $\Psi(H)$ has only $p_i/p_j'$-type and $c/c'$-type vertices, or $\Psi(H)$ has only $p_i/c'$-type and $c/p_j'$-type vertices.

		{\it Case 1.} Suppose first $\Psi(H)$ has only $p_i/p_j'$-type and $c/c'$-type vertices. We will denote $p_i/p_j'$-type vertices as $i/j$-type vertices for short.
		
		Let $v$ be an $i/j$-type vertex of $\Psi(H)$. Then all vertices adjacent to $v$ are of type $c/c'$. Let $f_1, \ldots,f_N$ be all edges incident to $v$, and let $f_s$ connect $v$ with $v_s$ for $s=1,\ldots,S$ (possibly some of the vertices $v_s$ coincide).
		Then, by Lemma \ref{together}, in the notations of which we have $t=m_i+1, \: t'=n_j + 1, \: r=1, \: r'=1$, we obtain
		\begin{equation}\label{s'}
		m_i \: \sum_{s=1}^S L(v_s)l_v(f_s) = n_j \: \sum_{s=1}^S L'(v_s)l_v(f_s),
		\end{equation}
		\begin{equation}\label{s'2}
		m_i \: \sum_{s=1}^S L(v_s)l'_v(f_s) = n_j \: \sum_{s=1}^S L'(v_s)l'_v(f_s).
		\end{equation}
		Now suppose that $w_1, \ldots, w_Q$ are all $c/c'$-type vertices of $\Psi(H)$. For every vertex $w_q, \: q=1,\ldots, Q$,  consider all the edges incident to $w_q$ which finish in $i/j$-type vertices (for fixed $i$ and $j$), and let $X^q_{ij}$ be the sum of the $l$-labels of these edges at the $i/j$-type ends, and let $Y^q_{ij}$ be the sum of the $l'$-labels of these edges at the $i/j$-type ends; we let $X^q_{ij}=Y^q_{ij}=0$ if there are no such edges.
		It follows from Lemma \ref{lll} that for all $q=1,\ldots, Q$ we have
		$$
		|F_{H,w_q}: P_{H,w_q}|=X^q_{i1}+X^q_{i2} + \ldots + X^q_{il}
		$$
		for every $i$, so
		\begin{equation}\label{qqq}
		X^q_{11}+X^q_{12} + \ldots + X^q_{1l} = X^q_{21}+X^q_{22} + \ldots + X^q_{2l}= \ldots=X^q_{k1}+X^q_{k2} + \ldots + X^q_{kl}. 		
		\end{equation}
		Analogously, for all $q=1, \ldots, Q$ if we denote $w'_q=\pi_*(w_q)$, then
		$$
		|F_{H',w'_q}: P_{H',w'_q}|=Y^q_{1j}+Y^q_{2j} + \ldots + Y^q_{kj}
		$$
		for every $j$, so
		\begin{equation}\label{qqq2}
		Y^q_{11}+Y^q_{21} + \ldots + Y^q_{k1} = Y^q_{12}+Y^q_{22} + \ldots + Y^q_{k2}= \ldots=Y^q_{1l}+Y^q_{2l} + \ldots + Y^q_{kl}. 		
		\end{equation}
		
		Now sum up the equalities of the form (\ref{s'}) for all $i/j$-type vertices $v$ (for fixed $i$ and $j$) and group the summands with the vertex label corresponding to the same $c/c'$-type vertex together.  We obtain
		$$	
		m_i \: \sum_{q=1}^Q L(w_q)X^q_{ij}= n_j \: \sum_{q=1}^Q L'(w_q)X^q_{ij}
		$$
		for all possible $i$ and $j$.
		Denote
		\begin{equation}\label{x}	
		x_{ij}=\sum_{q=1}^Q L(w_q)X^q_{ij}, \quad x'_{ij}=\sum_{q=1}^Q L'(w_q)X^q_{ij}.
		\end{equation}	
		Then we get
		\begin{equation}\label{X3}
		m_ix_{ij}=n_jx'_{ij}
		\end{equation}
		for all possible $i$ and $j$.
		 	
		Analogously, sum up the equalities of the form (\ref{s'2}) for all $i/j$-type vertices $v$ (for fixed $i$ and $j$) and group the summands with the vertex label corresponding to the same $c/c'$-type vertex together.  We obtain
		$$	
		m_i \: \sum_{q=1}^Q L(w_q)Y^q_{ij}= n_j \: \sum_{q=1}^Q L'(w_q)Y^q_{ij}
		$$
		for all possible $i$ and $j$.
		Denote
		\begin{equation}\label{x'}	
		y_{ij}=\sum_{q=1}^Q L(w_q)Y^q_{ij}, \quad y'_{ij}=\sum_{q=1}^Q L'(w_q)Y^q_{ij}.
		\end{equation}	
		Then we get
		\begin{equation}\label{X3'}
		m_iy_{ij}=n_jy'_{ij}
		\end{equation}
		for all possible $i$ and $j$.  	
		
		From (\ref{qqq}) and (\ref{x}) we get
		\begin{equation}\label{X1}
		x_{11}+x_{12} + \ldots + x_{1l} = x_{21}+x_{22} + \ldots + x_{2l}= \ldots=x_{k1}+x_{k2} + \ldots + x_{kl}, 		
		\end{equation}
		\begin{equation}\label{X2}
		x'_{11}+x'_{12} + \ldots + x'_{1l} = x'_{21}+x'_{22} + \ldots + x'_{2l}= \ldots=x'_{k1}+x'_{k2} + \ldots + x'_{kl}, 		
		\end{equation}
		for all possible $i$ and $j$.
		Summing the equalities (\ref{X3}) with the same $i$ we get the following system of equations
		\begin{equation}\label{S1}
		\begin{cases}
		m_1(x_{11}+ x_{12}+ \ldots + x_{1l})=n_1x_{11}'+n_2x_{12}'+\ldots +n_lx_{1l}' \\
		m_2(x_{21}+ x_{22}+ \ldots + x_{2l})=n_1x_{21}'+n_2x_{22}'+\ldots +n_lx_{2l}' \\
		\qquad \cdots \\
		m_{k-1}(x_{k-1,1}+ x_{k-1,2}+ \ldots + x_{k-1,l})=n_1x_{k-1,1}'+n_2x_{k-1,2}'+\ldots +n_lx_{k-1,l}' \\
		m_k(x_{k1}+ x_{k2}+ \ldots + x_{kl})=n_1x_{k1}'+n_2x_{k2}'+\ldots +n_lx_{kl}' \\
		\end{cases}
		\end{equation}
		Analogously, from (\ref{qqq2}) and (\ref{x'}) we get
		\begin{equation}\label{X1'}
		y_{11}+y_{21} + \ldots + y_{k1} = y_{12}+y_{22} + \ldots + y_{k2}= \ldots=y_{1l}+y_{2l} + \ldots + y_{kl}, 		
		\end{equation}
		\begin{equation}\label{X2'}
		y'_{11}+y'_{21} + \ldots + y'_{k1} = y'_{12}+y'_{22} + \ldots + y'_{k2}= \ldots=y'_{1l}+y'_{2l} + \ldots + y'_{kl}, 		
		\end{equation}
		for all possible $i$ and $j$.
		Summing the equalities (\ref{X3'}) with the same $j$ we get the following system of equations
		\begin{equation}\label{S1'}
		\begin{cases}
		m_1y_{11}+ m_2y_{21}+ \ldots + m_ky_{k1}=n_1(y_{11}'+y_{21}'+\ldots +y_{k1}') \\
		m_1y_{12}+ m_2y_{22}+ \ldots + m_ky_{k2}=n_2(y_{12}'+y_{22}'+\ldots +y_{k2}') \\
		\qquad \cdots \\
		m_1y_{1,l-1}+ m_2y_{2,l-1}+ \ldots + m_ky_{k,l-1}=n_{l-1}(y_{1,l-1}'+y_{2,l-1}'+\ldots +y_{k,l-1}') \\
		m_1y_{1l}+ m_2y_{2l}+ \ldots + m_ky_{kl}=n_{l-1}(y_{1l}'+y_{2l}'+\ldots +y_{kl}')  \\
		\end{cases}
		\end{equation}
		
		Let $A=(a_{ij})$ be the matrix with $k$ rows and $l$ columns, such that $a_{ij}=1$ if there are $i/j$-type vertices in $\Psi(H)$ and $a_{ij}=0$ otherwise. Note that for every $q=1, \ldots, Q$ we have $L(w_q)>0, \: L'(w_q)>0$, and $X^q_{ij}=0$ if and only if $Y^q_{ij}=0$, if and only if there are no vertices of type $i/j$ adjacent to $w_q$. Obviously every vertex of type $i/j$ of $\Psi(H)$ should be adjacent to at least one $c/c'$-type vertex, i.e. to at least one of $w_1,\ldots,w_Q$.
		This means that $x_{ij} \geq 0, \: x'_{ij} \geq 0$, and $x_{ij}=0$ if and only if $x'_{ij}=0$, if and only if $X^q_{ij}=0$ for all $q=1,\ldots,Q$, if and only if there are no vertices of type $i/j$ in $\Psi(H)$, if and only if $a_{ij}=0$. Analogously $y_{ij} \geq 0, \: y'_{ij} \geq 0$, and $y_{ij}=0$ if and only if $y'_{ij}=0$, if and only if $Y^q_{ij}=0$ for all $q=1,\ldots,Q$, if and only if there are no vertices of type $i/j$ in $\Psi(H)$, if and only if $a_{ij}=0$.
		
		Note also that for every $i$ there should exist some $j$ such that there are $i/j$-type vertices in $\Psi(H)$. Analogously, for every $j$ there should exist some $i$ such that $i/j$-type vertices exist in $\Psi(H)$. This means that the matrix $A$ does not have zero rows or columns.
		
		Combining (\ref{X1}), (\ref{X2}), (\ref{S1}), (\ref{X1'}), (\ref{X2'}), (\ref{S1'}), we have the following equations and conditions on $x_{ij}, x'_{ij}, y_{ij}, y'_{ij}, a_{ij}$:
		\begin{equation}\label{System1}
		\begin{cases}
		m_1(x_{11}+ x_{12}+ \ldots + x_{1l})=n_1x_{11}'+n_2x_{12}'+\ldots +n_lx_{1l}'  \\
		m_2(x_{21}+ x_{22}+ \ldots + x_{2l})=n_1x_{21}'+n_2x_{22}'+\ldots +n_lx_{2l}' \\
		\qquad \cdots \\
		m_{k-1}(x_{k-1,1}+ x_{k-1,2}+ \ldots + x_{k-1,l})=n_1x_{k-1,1}'+n_2x_{k-1,2}'+\ldots +n_lx_{k-1,l}' \\
		m_k(x_{k1}+ x_{k2}+ \ldots + x_{kl})=n_1x_{k1}'+n_2x_{k2}'+\ldots +n_lx_{kl}' \\
		x_{11}+x_{12} + \ldots + x_{1l} = x_{21}+x_{22} + \ldots + x_{2l}= \ldots=x_{k1}+x_{k2} + \ldots + x_{kl}
		\\
		x'_{11}+x'_{12} + \ldots + x'_{1l} = x'_{21}+x'_{22} + \ldots + x'_{2l}= \ldots=x'_{k1}+x'_{k2} + \ldots + x'_{kl}
		\end{cases}
		\end{equation}
		\begin{equation}\label{System2}
		\begin{cases}
		m_1y_{11}+ m_2y_{21}+ \ldots + m_ky_{k1}=n_1(y_{11}'+y_{21}'+\ldots +y_{k1}') \\
		m_1y_{12}+ m_2y_{22}+ \ldots + m_ky_{k2}=n_2(y_{12}'+y_{22}'+\ldots +y_{k2}') \\
		\qquad \cdots \\
		m_1y_{1,l-1}+ m_2y_{2,l-1}+ \ldots + m_ky_{k,l-1}=n_{l-1}(y_{1,l-1}'+y_{2,l-1}'+\ldots +y_{k,l-1}') \\
		m_1y_{1l}+ m_2y_{2l}+ \ldots + m_ky_{kl}=n_l(y_{1l}'+y_{2l}'+\ldots +y_{kl}')  \\
		y_{11}+y_{21} + \ldots + y_{k1} = y_{12}+y_{22} + \ldots + y_{k2}= \ldots=y_{1l}+y_{2l} + \ldots + y_{kl}
		\\
		y'_{11}+y'_{21} + \ldots + y'_{k1} = y'_{12}+y'_{22} + \ldots + y'_{k2}= \ldots=y'_{1l}+y'_{2l} + \ldots + y'_{kl} 	
		\end{cases}
		\end{equation}
		\begin{equation}\label{System3}
		x_{ij}=0 \Leftrightarrow x'_{ij}=0 \Leftrightarrow y_{ij}=0 \Leftrightarrow y'_{ij}=0 \Leftrightarrow a_{ij}=0, \quad i=1,\ldots,k, \: j=1, \ldots l
		\end{equation}
		\begin{equation}\label{System4}
		\forall i \: \forall j \: x_{ij}, x'_{ij}, y_{ij}, y'_{ij} \geq 0, \quad \forall i \: \exists j: x_{ij}>0, \quad \forall j \: \exists i: x_{ij}>0, \quad i=1,\ldots,k, \: j=1, \ldots l.
		\end{equation}
		
		Without loss of generality, we can assume that $k \geq l$.
		By $[\alpha]$ we denote the (lower) integer part of $\alpha$, so $[l/2]=l/2$ if $l$ is even, and $[l/2]=(l-1)/2$ if $l$ is odd.
		
		\begin{claim}
		In the above notation, $a_{ij}=a_{k+1-i, l+1-j}=0$ if $i \leq [l/2]$ or $j \leq [l/2]$, provided $i \neq j$.
		\end{claim}
	\begin{proof}
		We prove this claim by induction on $d_{ij}$, where $d_{ij} = \min \: (i, \: j)$. Thus we are interested in the case when $1 \leq d_{ij} \leq [l/2]$.
		
		The base of induction is the case when $d_{ij}=1$, so either $i=1$, or $j=1$. This case is included in the inductive step below, with $r=1$.
		
		Suppose the claim is proved for $d(i,j) \leq r-1$, where $1 \leq r \leq [l/2]$, and we want to prove it for $d(i,j)=r$.
		We have $a_{ij}=0$ if $i \leq r-1$ or $j \leq r-1$, provided $i \neq j$, and $a_{ij}=0$ if $i \geq k-r+2$ or $j \geq l-r +2$, provided $k-i \neq l-j$ (in the case $r=1$ there are no conditions).
		By (\ref{System3}), this means that the $r$-th equation of (\ref{System1}) has the following form
		\begin{equation}\label{z1}
		m_r(x_{rr}+ x_{r,r+1}+ \ldots + x_{r,l+1-r})=n_rx_{rr}'+n_{r+1}x_{r,r+1}'+\ldots +n_{l+1-r}x_{r,l+1-r}',
		\end{equation}	
		and the $(k+1-r)$-th equation of (\ref{System1}) has the following form
		\begin{multline}\label{z2}
		m_{k+1-r}(x_{k+1-r, r}+ x_{k+1-r,r+1}+ \ldots + x_{k+1-r,l+1-r})=\\=n_rx_{k+1-r,r}'+n_{r+1}x_{k+1-r,r+1}'+\ldots +n_{l+1-r}x_{k+1-r,l+1-r}'.
		\end{multline}  	
		Also from the last two lines of equations of (\ref{System1}) we obtain
		\begin{equation}\label{z3} 	
		x_{rr}+x_{r,r+1} + \ldots + x_{r,l+1-r}=x_{k+1-r,r}+x_{k+1-r,r+1} + \ldots + x_{k+1-r,l+1-r},
		\end{equation}
		\begin{equation}\label{z4}	
		x'_{rr}+x'_{r,r+1} + \ldots + x'_{r,l+1-r}=x'_{k+1-r,r}+x'_{k+1-r,r+1} + \ldots + x'_{k+1-r,l+1-r}.
		\end{equation}
		
		Since $n_r < n_{r+1} < \ldots < n_{l+1-r}$, (\ref{z1}) implies
		\begin{equation}\label{I1}
		m_r(x_{rr}+ x_{r,r+1}+ \ldots + x_{r,l+1-r}) \geq n_r(x_{rr}'+x_{r,r+1}'+\ldots +x_{r,l+1-r}')
		\end{equation}
		and equality in (\ref{I1}) is obtained if and only if $x_{r,r+1}'=x_{r,r+2}'=\ldots=x_{r,l+1-r}'=0$.
		
		Also (\ref{z2}) implies	
		\begin{multline}\label{I2}
		m_{k+1-r}(x_{k+1-r, r}+ x_{k+1-r,r+1}+ \ldots + x_{k+1-r,l+1-r}) \leq \\ \leq n_{l+1-r}(x_{k+1-r,r}'+x_{k+1-r,r+1}'+\ldots +x_{k+1-r,l+1-r}'),
		\end{multline}  	
		and equality in (\ref{I2}) is obtained if and only if $x_{k+1-r,r}'=x_{k+1-r,r+1}'=\ldots =x_{k+1-r,l-r}'=0$.
		Multiplying (\ref{I1}) by $m_{k+1-r}$, (\ref{I2}) by $m_r$ and using (\ref{z3}), (\ref{z4}), we obtain
		\begin{multline}\label{long}
		m_{k+1-r}n_r(x_{rr}'+x_{r,r+1}'+\ldots +x_{r,l+1-r}') \leq m_{k+1-r} m_r(x_{rr}+ x_{r,r+1}+ \ldots + x_{r,l+1-r}) = \\ =m_r m_{k+1-r}(x_{k+1-r, r}+ x_{k+1-r,r+1}+ \ldots + x_{k+1-r,l+1-r}) \leq \\
		\leq m_r n_{l+1-r}(x_{k+1-r,r}'+x_{k+1-r,r+1}'+\ldots +x_{k+1-r,l+1-r}')= \\ =m_r n_{l+1-r}(x_{rr}'+x_{r,r+1}'+\ldots +x_{r,l+1-r}').
		\end{multline}
		Note that by (\ref{System3}), (\ref{System4}) we have $x_{rr}'+x_{r,r+1}'+\ldots +x_{r,l+1-r}' \neq 0$ (since otherwise the $r$-th row of $A$ is zero, which is impossible). So (\ref{long}) implies
		\begin{equation}\label{oo}	
		m_{k+1-r}n_r \leq m_r n_{l+1-r}.
		\end{equation}	
		
		Now apply analogous arguments to (\ref{System2}). By induction hypothesis the $r$-th equation of (\ref{System2}) has the following form
		\begin{equation}\label{z1'}
		m_ry_{rr}+ m_{r+1}y_{r+1,r}+ \ldots + m_{k+1-r}y_{k+1-r,r}=n_r(y_{rr}'+y_{r+1,r}'+\ldots +y_{k+r-1,r}'),
		\end{equation}	
		and the $(l+1-r)$-th equation of (\ref{System2}) has the following form
		\begin{multline}\label{z2'}
		m_ry_{r,l+1-r}+ m_{r+1}y_{r+1,l+1-r}+ \ldots + m_{k+1-r}y_{k+1-r,l+1-r}= \\ =n_{l+1-r}(y_{r,l+1-r}'+y_{r+1,l+1-r}'+\ldots +y_{k+r-1,l+1-r}').
		\end{multline}  	
		Also from the last two lines of equations (\ref{System2}) we obtain
		\begin{equation}\label{z3'} 	
		y_{rr}+y_{r+1,r} + \ldots + y_{k+1-r,r}=y_{r,l+1-r}+y_{r+1,l+1-r} + \ldots + y_{k+1-r,l+1-r},
		\end{equation}
		\begin{equation}\label{z4'}	
		y'_{rr}+y'_{r+1,r} + \ldots + y'_{k+1-r,r}=y'_{r,l+1-r}+y'_{r+1,l+1-r} + \ldots + y'_{k+1-r,l+1-r}.
		\end{equation}
		
		Since $m_r < m_{r+1} < \ldots < m_{k+1-r}$, (\ref{z1'}) implies
		\begin{equation}\label{I1'}
		m_r(y_{rr}+ y_{r+1,r}+ \ldots + y_{k+1-r,r}) \leq n_r(y_{rr}'+y_{r+1,r}'+\ldots +y_{k+r-1,r}'),
		\end{equation}
		and equality in (\ref{I1'}) is obtained if and only if $y_{r+1,r}=y_{r+2,r}=\ldots=y_{k+1-r,r}=0$.

		Also (\ref{z2'}) implies	
		\begin{multline}\label{I2'}
		m_{k+1-r}(y_{r,l+1-r}+ y_{r+1,l+1-r}+ \ldots + y_{k+1-r,l+1-r}) \geq \\ \geq n_{l+1-r}(y_{r,l+1-r}'+y_{r+1,l+1-r}'+\ldots +y_{k+r-1,l+1-r}'),
		\end{multline}
		and equality in (\ref{I2'}) is attained if and only if $y_{r,l+1-r}=y_{r+1,l+1-r}=\ldots=y_{k-r,l+1-r}=0$.
		Multiplying (\ref{I1'}) by $n_{l+1-r}$, (\ref{I2'}) by $n_r$ and using (\ref{z3'}), (\ref{z4'}), we obtain
		\begin{multline}\label{long'}
		n_{l+1-r}m_r(y_{rr}+ y_{r+1,r}+ \ldots + y_{k+1-r,r}) \leq n_{l+1-r}n_r(y_{rr}'+y_{r+1,r}'+\ldots +y_{k+r-1,r}')= \\ = n_rn_{l+1-r}(y_{r,l+1-r}'+y_{r+1,l+1-r}'+\ldots +y_{k+r-1,l+1-r}') \leq \\ \leq n_rm_{k+1-r}(y_{r,l+1-r}+ y_{r+1,l+1-r}+ \ldots + y_{k+1-r,l+1-r})=\\=n_rm_{k+1-r}(y_{rr}+ y_{r+1,r}+ \ldots + y_{k+1-r,r}).
		\end{multline}
		Note that by (\ref{System3}), (\ref{System4}) we have $y_{rr}+ y_{r+1,r}+ \ldots + y_{k+1-r,r} \neq 0$ (since otherwise the $r$-th column of $A$ is zero, which is impossible). So (\ref{long'}) implies $$n_{l+1-r}m_r \leq n_rm_{k+1-r}.$$
		Combined with (\ref{oo}), this gives
		\begin{equation}\label{tratata}		
		 m_{k+1-r}n_r = m_r n_{l+1-r},
		\end{equation}
		so all the inequalities in (\ref{long}), (\ref{long'}), and then also in (\ref{I1}), (\ref{I2}), (\ref{I1'}), (\ref{I2'}), turn into equalities. As mentioned above, this means that
		$$
		x_{r,r+1}'=x_{r,r+2}'=\ldots=x_{r,l+1-r}'=0, \quad x_{k+1-r,r}'=x_{k+1-r,r+1}'=\ldots =x_{k+1-r,l-r}'=0,
		$$
		so by (\ref{System3}) we have
		$$	
		a_{r,r+1}=a_{r,r+2}=\ldots=a_{r,l+1-r}=0, \quad a_{k+1-r,r}=a_{k+1-r,r+1}=\ldots =a_{k+1-r,l-r}=0,
		$$
		and 	
		$$	
		y_{r+1,r}=y_{r+2,r}=\ldots=y_{k+1-r,r}=0, \quad y_{r,l+1-r}=y_{r+1,l+1-r}=\ldots=y_{k-r,l+1-r}=0,
		$$
		so by (\ref{System3}) we have
		$$	
		a_{r+1,r}=a_{r+2,r}=\ldots=a_{k+1-r,r}=0, \quad a_{r,l+1-r}=a_{r+1,l+1-r}=\ldots=a_{k-r,l+1-r}=0.
		$$
		Combined with the induction hypothesis, this proves that
		$a_{ij}=a_{k+1-i, l+1-j}=0$ if $i \leq r$ or $j \leq r$, provided $i \neq j$. This proves the claim.
	\end{proof}

		Thus we have proved that
		\begin{equation}\label{zero}
		a_{ij}=a_{k+1-i, l+1-j}=0, \: \text{if} \: i \leq [l/2] \: \text{or} \: j \leq [l/2], \: \text{provided} \: i \neq j.
		\end{equation}
		
		\begin{claim}
			In the above notation, one has $k=l$.
		\end{claim}
		\begin{proof}
			Suppose, on the contrary, that $k>l$. Suppose first $l$ is even. Then we have from (\ref{zero}) that $a_{ij}=0$ if $j \leq l/2$ and $i \neq j$, and $a_{ij}=0$ if $j \geq l/2+1$ and $k-i \neq l-j$. In particular, taking $i=l/2+1$, we get that $a_{l/2+1,j}=0$ for $j \leq l/2$ (since for these pairs $i \neq j$), and  $a_{l/2+1,j}=0$ for $j \geq l/2+1$ (since for these pairs $k-i=k-l/2-1 > l/2 - 1 \geq l -j$), so $A$ has a zero row, a contradiction.
		
		So $l$ is odd. Then it follows from (\ref{zero}) that $a_{ij}=0$ if $j \leq (l-1)/2$ and $i \neq j$, and $a_{ij}=0$ if $j \geq (l+3)/2$ and $k-i \neq l-j$. In particular, $a_{(l+1)/2,j}=a_{(l+3)/2,j}=0$ when $j \neq (l+1)/2$. Denote $(l+1)/2=b$, then $(l+3)/2=b+1$. Due to (\ref{System3}), this means that among the equations (\ref{System1}) we have the following:
		$$
		\begin{cases}
		m_{b}x_{bb}=n_{b}x'_{bb}, \\m_{b+1}x_{b+1,b}=n_{b}x'_{b+1,b},  \\
		x_{bb}=x_{b+1,b}, \\
		x_{bb}'=x_{b+1,b}'.
		\end{cases}
		$$
		This immediately implies $m_b=m_{b+1}$, a contradiction. This shows that $k=l$.
		\end{proof}
	
		It follows from (\ref{zero}) that $a_{ij}=0$ if $i \neq j$. Then $A$ is the identity matrix, since it does not have zero rows or columns, so, according to (\ref{System3}), (\ref{System4}), we have $x_{ij}=x'_{ij}=0$ if $i \neq j$, and $x_{ii}>0$ for all $i=1,\ldots,k$, therefore (\ref{System1}) turns into
		$$
		\begin{cases}
		m_1x_{11}=n_1x_{11}' \\
		m_2x_{22}=n_2x_{22}' \\
		\qquad \cdots  \\
		m_kx_{kk}=n_kx_{kk}' \\
		x_{11}=x_{22}=\ldots=x_{kk} \\
		x'_{11}=x'_{22}=\ldots=x'_{kk},		
		\end{cases}
		$$
		which immediately implies (\ref{mn}). This finishes Case 1.

		{\it Case 2.}
		Now suppose that $\Psi(H)$ has only $p_i/c'$-type and $c/p_j'$-type vertices, $i=1,\ldots,k$, $j=1,\ldots,l$. The proof in this case is almost the same as the proof of Case 2 in Theorem \ref{th2}.
		
		Let $v$ be a $c/p_j'$-type vertex of $\Psi(H)$, $j=1,\ldots,l$. Then all vertices adjacent to $v$ are of types $p_i/c'$, $i=1,\ldots,k$. Let $f_1, \ldots f_S$ be all edges incident to $v$, and let $f_s$ connect $v$ with $v_s$ for $s=1,\ldots,S$ (possibly some of the vertices $v_s$ coincide).
		Then, by Lemma \ref{together}, in the notations of which we have $t = k, \: t' = n_j +1, \: r = k, \: r' = 1$, we obtain
		\begin{equation}\label{s2'}
		\: \sum_{s=1}^S L(v_s)l'_v(f_s) = \frac{k}{k-1}n_j \: \sum_{s=1}^S L'(v_s)l'_v(f_s).
		\end{equation}

		Now suppose that $w_1, \ldots, w_Q$ are all $p_i/c'$-type vertices of $\Psi(H)$, for all $i=1,\ldots,k$.  For every vertex $w_q, \: q=1,\ldots,Q$ and $j=1,2$, consider all the edges incident to $w_q$ which finish in $c/p_j'$-type vertices (note that such edges should always exist), and let $Z^q_j$ be the sum of the $l'$-labels of these edges at the $c/p_j'$-type ends.  Then we have $Z^q_j > 0$ for every $q=1,\ldots, Q$ and $j=1,2$. Let $w_q'=\varphi_*(w_q)$ for  $q=1,\ldots, Q$. Then it follows from Lemma \ref{lll} that for every $q=1,\ldots,Q$ we have
		\begin{equation}\label{kkk'}	
		|F_{H', w'_q}: P_{H', w'_q}|=Z^q_1=Z^q_2.
		\end{equation}
		Now sum up the equalities of the form (\ref{s2'}) for all $c/p_1'$-type vertices $v$, and group the summands with the vertex label corresponding to the same $p_i/c'$-type vertex together.  We obtain
		\begin{equation}\label{rr1}
		\sum_{q=1}^Q L(w_q)Z^q_1= \frac{k}{k-1}n_1 \: \sum_{q=1}^Q L'(w_q)Z^q_1.
		\end{equation}
		Analogously, summing the equalities of the form (\ref{s2'}) for all $c/p_2'$-type vertices $v$ we get
		\begin{equation}\label{rr2}
		\sum_{q=1}^Q L(w_q)Z^q_2=  \frac{k}{k-1}n_2 \: \sum_{q=1}^Q L'(w_q)Z^q_2.
		\end{equation}
		Together with (\ref{kkk'}), equalities (\ref{rr1}) and (\ref{rr2}) imply that $n_1=n_2$, a contradiction.
		
		This finishes the proof of the first claim of the theorem.
		
		The proof of the second claim is similar, but easier. Indeed, let $m_1=m_2=m$, $k=2$, and apply the arguments as in the proof of the first claim above.	In Case 1, the above argument shows that (\ref{tratata}) still holds with $r=1$, so $m_2n_1=m_1n_l$, but $m_1=m_2$, so $n_1=n_l$, a contradiction. The proof of Case 2 is the same as above.	
		
		This proves Theorem \ref{th3}.
	\end{proof}

	\section{Characterisation of commensurability classes of RAAGs defined by trees of diameter 4}\label{Sec:com}

We now turn to prove commensurability of some RAAGs defined by trees of diameter 4.
\begin{prop} \label{prop:1}
	Let $\GG=\GG(T)$, where $T=T((v_1,1),\dots, (v_l,1);0)$, $l \geq 2$. For any $k_1,\dots, k_l$ there exist a non-negative integer $q$ and a finite index subgroup $H$ of $\GG$ such that $H\simeq \GG(S)$, where $S$ is a tree of diameter 4 of the form $S=T((v_1,k_1),\dots, (v_l,k_l); q)$. In particular, $\GG(S)$ and $\GG(T)$ are commensurable.
	
	 If $\GG=\GG(T)$, where $T=T((d,2);0)$, then for any $k>2$ there exists a finite index subgroup $H$ of $\GG$ such that $H\simeq \GG(S)$, where $S=T((d,k);0)$.
\end{prop}
\begin{proof}
	We first consider the case when $l=1$, i.e. $T=T((d,2);0)$. Let $p_1,p_2$ be the pivots of $T$. Let $\phi$ be the epimorphism $\GG\to \BZ_k$, induced by the map $p_1\mapsto 1$ and $x\mapsto 0$, where $x$ is any canonical generator of $\GG$ different from $p_1$. It is not hard to see that $\ker \phi$ is isomorphic to $\GG(S)$, see \cite{KK,BN}.
	
	Suppose now $l \geq 2$. Denote the only pivot of valency $v_i$ in $T$ by $p_i, \: i=1, \ldots, l$, and the central vertex of $T$ by $c$. Let $V(T)=X$. Let $F(X)$ be the free group with basis $X$. Then there is a natural epimorphism $\phi: F(X)\to \GG$. Fix $k_1,\dots, k_l$ and without loss of generality assume that $k_1\le k_2\le \dots\le k_l$. We can always suppose that $k_l > 1$, since otherwise $k_1=k_2=\ldots=k_l=1$, and the claim is obvious. Let $x_{i,1}, \ldots, x_{i,v_i}$ be all leaves of $T$ adjacent to the pivot $p_i$, $i=1, \ldots, l$.

	Consider the finite index subgroup $G<F(X)$ defined as the subgroup corresponding to the finite cover of the bouquet of $|X|$ circles $K$ defined as follows. Below by  $(a_1,\ldots,a_{k_l})$ we mean a simple cycle of length $k_l$, with vertices appearing in the order $a_1,\ldots,a_{k_l},a_1$, and analogously with other cycles. We take two cycles: the first one of the form $(a_1,\dots, a_{k_l})$, of length $k_l$, with all edges labelled by $p_1$; and the second one of the form $(b_1, \dots, b_{k_1})$, of length $k_1$, with all edges labelled by $p_l$; we identify these cycles by a vertex, $a_1=b_1$. This vertex is the basepoint of the based cover we construct. The degree of the finite cover of the bouquet of $|X|$ circles we construct is $k_1+k_l-1$, so no new vertices will be added, only edges. We first complete the constructed graph to a cover of the free group $F(p_1,p_l)$, by adding loops labelled by $p_1$ at vertices $b_2,\ldots, b_{k_1}$ (no loops are added if $k_1=1$), and adding loops labelled by $p_l$ at vertices $a_2, \ldots, a_{k_l}$.
	
	If $l=2$, then we are done with the construction of the cover. If $l > 2$, we consider two cases. In the first case we suppose that $k_2 = 1$, so for some $2 \leq j < l$ we have $1=k_1=k_2=\dots =k_j<k_{j+1}\le k_{j+2}\le \dots \le k_l$. In the second case we assume that $k_2 > 1$.
	
	\begin{figure}[!h]
		\centering
		\includegraphics[keepaspectratio,width=3.2in]{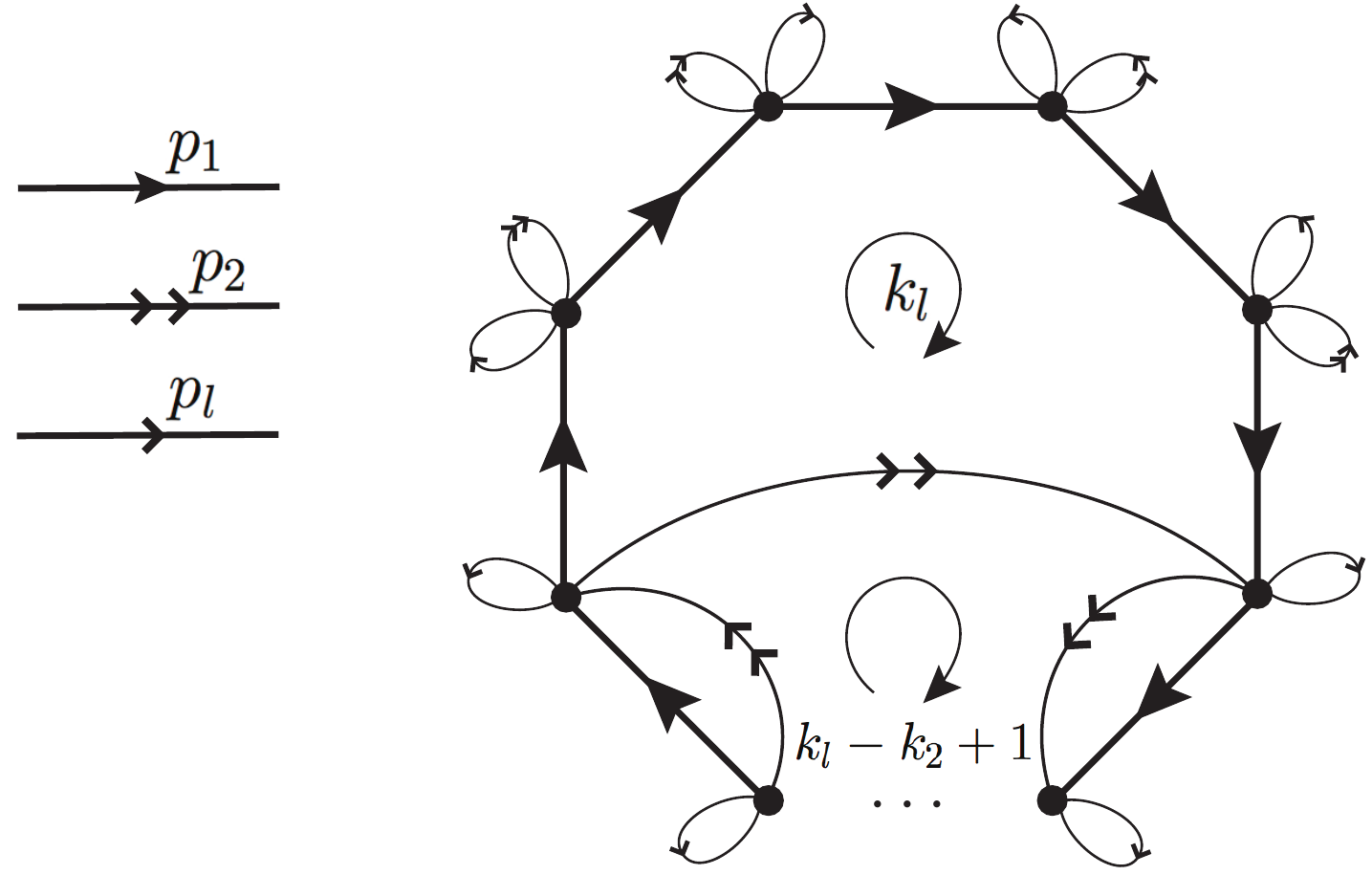}
		\caption{Constructing the cover, case 1} \label{fig:graph}
	\end{figure}
	
	In the first case, for every $i=2,\dots, j$, we add a cycle $(a_1,\dots, a_{k_l})$ labelled by $p_i$. For every $i=j+1,\dots, l-1$, we add loops labelled by $p_i$ to vertices $a_1,\dots, a_{k_i-1}$, and add a cycle $(a_{k_i},a_{k_i+1}, \dots, a_{k_l})$, of length $k_l-k_i+1$, with all edges labelled by $p_i$, see Figure \ref{fig:graph}. Adding loops labelled by the generators from $X \setminus \{p_1,\dots, p_l\}$ at all vertices of the graph we obtain a finite cover of the bouquet of $|X|$ circles and hence this defines a finite index subgroup $G$ of $F(X)$.
	
	\begin{figure}[!h]
		\centering
		\includegraphics[keepaspectratio,width=3.2in]{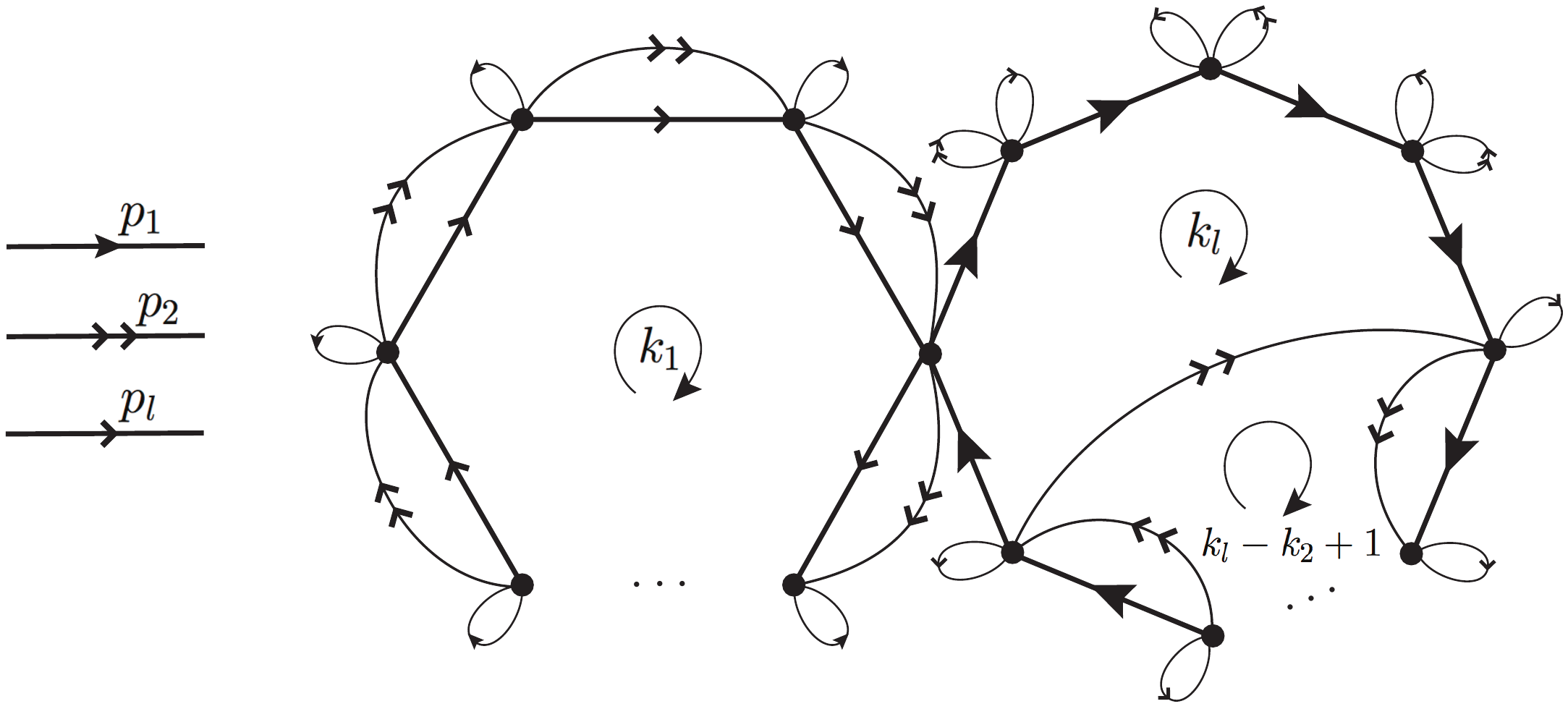}
		\caption{Constructing the cover, case 2} \label{fig:graph2}
	\end{figure}
	
	In the second case, for every $i=2,\dots, l-1$, we do the following. First add a cycle $(b_1,\dots, b_{k_1})$ with all edges labelled  by $p_i$. Then add loops labelled by $p_i$ to vertices $a_2,\dots, a_{k_i-1}$.
	Finally, add a cycle $(a_{k_i},a_{k_{i}+1}, \dots, a_{k_l})$, of length $k_l-k_i+1$, labelled by $p_i$, see Figure \ref{fig:graph2}. Adding loops labelled by the generators $X \setminus \{p_1,\dots, p_l\}$ at all vertices of the graph we obtain a finite cover of the bouquet of $|X|$ circles and  hence this defines a finite index subgroup $G$ of $F(X)$.
	
	We continue under the assumptions of the second case. Proof in the first case is analogous.	
	
	Note that $G$ has index $k_1+k_l-1$ in $F(X)$ by construction.	
		We have epimorphisms $\phi: F(X)\to \GG$ and $\psi: \GG\to F(p_1,\dots, p_l)$, where $F(p_1,\dots, p_l)$ is the free group with the basis $p_1,\ldots,p_l$. By construction, the image of $\psi(\phi(G))$ has index $k_1+k_l-1$ in $F(p_1,\dots, p_l)$. Hence,
	$$
	[F(X):G]\ge [\GG:\phi(G)]\ge [F(p_1,\dots, p_l):\psi(\phi(G))],
	$$
	and it follows that $H=\phi(G)$ has index $k_1+k_l-1$ in $\GG$. 	
	
	Choosing the maximal subtree spanned by the edges $\{(a_i,a_{i+1}), i=1,\dots, k_l-1 \}$, labelled by $p_l$, and $\{ (b_j,b_{j+1})\mid  \ j=1,\dots, k_1-1\}$, labelled by $p_1$, and using Nielsen transformations, it is not difficult to see that $G$ has a free basis which consists of the following generators, where the elements in $A$ form a free basis of $G \cap F(p_1, \ldots, p_l)$:
	\begin{enumerate}
		\item[$A:$]
		\begin{itemize}
			\item  $\{p_i^{k_1}, p_i^{(p_1^{j_i-1})}, {\left(p_i^{k_l-k_i+1}\right)}^{p_1^{k_i-1}} \mid j_i=2,\dots, k_i-1, \: i=2,\dots, l-1\}$;
			\item  $\{p_1^{k_l}, {p_1}^{(p_l^{j_1-1})}\mid j_1=2,\dots, k_1\}$;
			\item  $\{p_l^{k_1}, p_l^{(p_1^{j_l-1})} \mid j_l=2,\dots, k_l \}$.
			\item  $q$ other generators $h_1, \ldots, h_q$ which are words in $p_1, \ldots, p_l$, each of them containing at least two of $p_1, \ldots, p_l$.
		\end{itemize}
		\item[$B$:] $\{x_{i,m_i}^{p_1^r}\mid r=0,\dots, k_l-1\}$, $i =1, \ldots, l$, $m_i=1, \ldots, v_i$; \\
$\{x_{i,m_i}^{p_l^s}\mid s=1,\dots, k_1-1\}$, $i =1, \ldots, l$, $m_i=1, \ldots, v_i$.
		\item [$C$:] $\{c^{p_1^r}\mid r=0,\dots, k_l-1\}, \: \{c^{p_l^s}\mid s=1,\dots, k_1-1\}$.
	\end{enumerate}
	Here $q=(k_1+k_l-1)(l-1)+1-(k_1+\ldots+k_l)$. Indeed, this can be seen by direct calculations or by applying Schreier formula: $G \cap F(p_1, \ldots, p_l)$ is a subgroup of index $k_1+k_l-1$ in the free group $F(p_1, \ldots, p_l)$ of rank $l$, so by Schreier formula  $rk(G \cap F(p_1, \ldots, p_l))=(k_1+k_l-1)(l-1)+1$, and there are $k_1+\ldots+k_l$ generators in $A$ of the first three types, so the above formula for $q$ follows.
	
	\par
Thus $H$ is a finite index subgroup of $\GG$, and $H$ is generated by the set $\phi(A) \cup \phi(B) \cup \phi(C)$. Note that in $\GG$ all elements of $C$ become equal to $c$, so $\phi(C)$ consists just of one generator $c$.
	
	We now show that $H\simeq \GG(S)$, where $S=T((v_1,k_1),\dots, (v_l,k_l); q)$.
Note that the group $\GG$ splits as a fundamental group of a star of groups: $\GG=\pi_1(\mathbb{B})$, with $l$ leaves, with vertex groups at leaves equal to $B_i= \langle c, p_{i}, x_{i,1},\dots, x_{i,v_i}\rangle$, $i=1,\dots, l$, and vertex group at the center vertex equal to $B_c = \langle c, p_1, \ldots, p_l \rangle$. The edge groups are $E_i=\langle c, p_i \rangle$, $i=1,\ldots, l$.	This is the reduced centralizer splitting of $\mathbb{G}$. Let $\mathbf{T}$ be the Bass-Serre tree corresponding to this splitting of $\GG$.

	Recall that by Bass-Serre theory vertices of $\mathbf{T}$ correspond to left cosets by the vertex groups of $\mathbb{B}$, and edges of $\mathbf{T}$ correspond to left cosets by the edge groups of $\mathbb{B}$.
	Consider the finite subgraph $Y_0$ of the tree $\mathbf{T}$, which consists of the vertex $B_c$ and the incident edges $p_{1}^{j_i}E_i$ going to the vertices $p_{1}^{j_i}B_i$, $j_i=0,\dots, k_i-1, \: i=2,\dots, l$, and  edges $p_{l}^{j_1}E_1$ going to the vertices $p_{l}^{j_1}B_1$, $j_1=0,\dots, k_1-1$, together with these vertices. Note that all these vertices are indeed different. Thus $Y_0$ as a graph is a star with $k_1+\dots +k_l$ leaves.
	 We claim that $Y_0$ is the fundamental domain under the action of $H$ on $\mathbf{T}$, i.e., $Y_0$ contains exactly one representative of each vertex and edge orbit under the action of $H$ on $\mathbf{T}$.

	 First we show that $Y_0$ cannot contain two vertices or edges in the same orbit under the action of $H$. It suffices to show this for vertices of $Y_0$.
	 Note that $B_c$ is not in the same $H$-orbit with any other vertex of $Y_0$, since it is not even in the same $\GG$-orbit. Suppose that $u=p_{1}^{j_i}B_i$ is in the same $H$-orbit with some other vertex $v$ of $Y_0$, then $v$ can only be of the form $v=p_{1}^{j_i'}B_i$, where $j_i' \neq j_i$; here $j_i,j_i'=0,\dots, k_i-1, \: i=2,\dots, l$. Then there exists some $d \in B_i=\langle c, p_{i}, x_{i,1},\dots, x_{i,v_i}\rangle$ such that
    $$
    H p_1^{j_i} d = H p_1^{j_i'}.
    $$
    If $j_i < k_i-1$, then from the definition of $H$, see Figure \ref{fig:graph2}, we have that
    $$
    H p_1^{j_i} B_i = H p_1^{j_i},
    $$
    and it follows that $H p_1^{j_i'}=H p_1^{j_i}$, so $p_1^{j_i'-j_i} \in  H$, which is a contradiction by definition of $H$.

    If $j_i = k_i-1$, then $j_i' < k_i - 1$, and we get a contradiction in the same way as above.

    In the same way we can prove that the vertex $p_{l}^{j_1}B_1$, $j_1=0,\dots, k_1-1$, is not in the same $H$-orbit with some other vertex in $Y_0$. Thus indeed $Y_0$ does not contain two vertices or edges in the same orbit under the action of $H$.

	 Now we show that $Y_0$ contains at least one representative of each orbit, and so $Y_0$ is a fundamental domain. Note that it suffices to show that every edge of $\mathbf{T}$ which is incident to some edge in $Y_0$ can be taken to some edge in $Y_0$ by an element of $H$. Let $e$ be some edge of $\mathbf{T}$ incident to the vertex $B_c$, then by Bass-Serre theory $e=gE_i$, where $g \in B_c$, and we can suppose $g$ does not contain $c$, so $g=g(p_1,\ldots,p_l)$. Suppose $i > 1$, the case $i=1$ is analogous. By definition of $H$, see Figure \ref{fig:graph2},
	  $$Q=\{ p_1^s, \: s=0, \ldots, k_i-1; \: p_1^{k_i-1}p_i^m, \: m=1, \ldots,k_l-k_i  ;\: p_l^t, \: t=1, \ldots k_1-1 \}$$
	  is the set of coset representatives of
	  $H$ in $\mathbb{G}$, so $g=hq$ for some $q \in Q$, and $e=gE_i$ is in the same $H$-orbit as $qE_i$. Since $p_i \in E_i$, by definition of $Y_0$ the edge $qE_i$ is in $Y_0$, and so we are done.

	 Now, if $f$ is some edge of  $\mathbf{T}$ incident to the vertex $p_{1}^{j_i}B_i$, $j_i=0,\dots, k_i-1, \: i=2,\dots, l$, then by Bass-Serre theory $f=gE_i$, where $g = p_{1}^{j_i}b, \: b \in B_i$, so $f=  p_{1}^{j_i}bE_i$, and we can suppose $b$ does not contain $p_i$, so $b=w(c, x_{i,1},\dots, x_{i,v_i})$. Then, by definition of $H$, $p_{1}^{j_i}b(p_{1}^{j_i})^{-1}=h \in H$, so $f=p_{1}^{j_i}bE_i=hp_{1}^{j_i}E_i$ is in the same $H$-orbit as $p_{1}^{j_i}E_i$, which is in $Y_0$ by definition. The case when $f$ is some edge of  $\mathbf{T}$ incident to the vertex  $p_{l}^{j_1}B_1$, $j_1=0,\dots, k_1-1$ is analogous.
	
	  This proves that $Y_0$ is the fundamental domain under the action of $H$ on $\mathbf{T}$.
Let $H \cong \pi_1(\mathbb{Y})$ be the induced splitting of $H$ as a fundamental group of a graph of groups, corresponding to the induced action of $H$ on $\mathbf{T}$. Let $Y$ be the underlying graph of $\mathbb{Y}$. Let $\pi: \mathbf{T} \rightarrow Y$ be the natural projection morphism.
	
	   It follows that the morphism $\pi$ restricted to $Y_0$ induces an isomorphism of graphs $\pi_{Y_0}: Y_0 \rightarrow Y$, and the vertex group at $\pi(v)$ is equal to the stabilizer of the vertex $v$ for every vertex $v$ of $Y_0$.
This means that $\mathbb{Y}$ is a star with $k_1+\dots +k_l$ leaves, $u_{1,1},\dots, u_{1,k_1}; \dots; u_{l,1},\dots, u_{l,k_l}$ and the center vertex $z$. The vertex group at $z$ is $H\cap \langle c, p_1, \ldots, p_l \rangle=\langle \hbox{generators of type} \: A, c\rangle$. The vertex groups $G_{u_{i, j_i}}$ at leaves $u_{i,j_i}$ are the following:
	$$
	\left\{
	\begin{array}{ll}
	G_{u_{i, j_i}}=H \cap B_i^{(p_{1}^{j_i})}= H \cap \langle c, p_i, x_{i,1},\dots, x_{i,v_i}\rangle^{p_{1}^{j_i}}, &  j_i=0,\dots, k_i-1, \: i=2,\dots, l;\\
	G_{u_{1, j_1}}=H \cap B_1^{(p_{l}^{j_1})}= H \cap \langle c, p_1, x_{1,1},\dots, x_{1,v_1}\rangle^{p_{l}^{j_1}}, &  j_1=0,\dots, k_1-1.
	\end{array}
	\right.
	$$
	
	The edge groups of $\mathbb{Y}$ are the corresponding intersections of vertex groups. Since $c \in H$ it follows that the vertex groups are direct products of a free group and the infinite cyclic group. Moreover, since, by construction, conjugates of $x_{i,1},\dots, x_{i,v_i-1}$ belong to $H$, it follows that every edge group of $\mathbb{Y}$ is the direct product of the cyclic group generated by $c$ and a subgroup of $\langle p_{1},\dots, p_{l}\rangle$.
	
	Let $S=T((v_1,k_1),\ldots, (v_l,k_l);q)$, where $q$ is as above. Let $p_{i,1},\ldots,p_{i,k_i}$ be the pivots of $T$ of valency $v_i+1$, $i=1, \ldots, l$. Let $x_{i, j_i, 1}, \ldots, x_{i, j_i, v_i}$ be all leaves of $S$ which are adjacent to the pivot $p_{i,j_i}$, for $j_i=1,\ldots,k_i$, $i=1,\ldots, l$. Denote the center of $S$ by $c'$, and the hairs of $S$ by $y_1, \ldots, y_q$.		
	
	It follows that the following map induces  an isomorphism from $\pi_1(\mathbb{Y})$ to $\GG(S)$:
	$$
	\begin{array}{l}
	c \mapsto c',  \\  h_i \mapsto y_i , \: i=1, \ldots, q, \\
	p_1^{k_l} \mapsto p_{1,1}  \\ p_1^{(p_l^{j_1-1})} \mapsto p_{1,j_1}, \: j_1=2, \ldots, k_1
	\\
	p_{i}^{k_1} \mapsto p_{i,1}, \: i=2, \ldots, l,  \\  p_i^{(p_1^{j_i-1})} \mapsto p_{i,j_i}, \: j_i=2, \ldots, k_i-1, \: i=2, \ldots, l  \\
	{\left(p_i^{k_l-k_i+1}\right)}^{p_1^{k_i-1}}  \mapsto p_{i, k_i}, \: i=2, \ldots, l
	\\
	x_{i,m_i}^{p_1^{r_i}} \mapsto  x_{i, r_i+1, m_i},    \: r_i=0,\dots, k_i-1, \:  i =1, \ldots, l$, $m_i=1, \ldots, v_i \\
	x_{1,m_1}^{p_l^s} \mapsto 	x_{1, s+1, m_1}, \: s=0,\dots, k_1-1, \: m_1 = 1, \ldots, v_1.
	\end{array}
	$$
	 In other words, this means that $\pi_1(\mathbb{Y})$ is the reduced centralizer splitting of $\GG(S)$. We conclude that $H\simeq \GG(S)$.
\end{proof}

\begin{prop} \label{prop:2}
	Let $T=T((v_1,k_1),\dots, (v_l,k_l);p)$ and $S=T((v_1,k_1),\dots, (v_l,k_l); q)$ be two trees of diameter 4. Then there exists $r$ such that $\GG(R)$ is a finite index subgroup of both $\GG(S)$ and $\GG(T)$, where $R=T((v_1,k_1),\dots, (v_l,k_l);r)$. In particular, $\GG(S)$ and $\GG(T)$ are commensurable.
\end{prop}
\begin{proof}
	Let $K=k_1+\dots+ k_l$. Since the group $\GG(S)$ (the group $\GG(T)$ correspondingly) retracts onto the free subgroup $F_{K+q}$ ($F_{K+p}$ correspondingly) generated by pivots and hair vertices, the full preimage of a subgroup of $F_{K+q}$ (of $F_{K+p}$ correspondingly) of finite index  $I$ is an index $I$ subgroup of $\GG(S)$ ($\GG(T)$, correspondingly).
	
	We define subgroups of $F_{K+q}$ and $F_{K+p}$ via covers of bouquet of $K+q$ and $K+p$ circles correspondingly. The subgroup $A_q$ of $F_{K+q}$ is defined as follows. Take $p+K-1$ points $a_1,\dots, a_{p+K-1}$ and for every pivot generator of $F_{K+q}$  add a cycle $(a_1,\dots, a_{p+K-1})$ of length $p+K-1$ labelled by this generator. We complete the obtained graph to a cover of $F_{K+q}$ by adding $q$ loops labelled by the hair generators of $F_{K+q}$ at every vertex $a_i$, $i=1,\dots, p+K-1$. The subgroup $A_p$ of $F_{K+p}$ is defined in a similar fashion.
	
	We note that $A_q$ and $A_p$ are free subgroups of index $p+K-1$ and $q+K-1$ in $F_{K+q}$ and $F_{K+p}$ correspondingly, and both have rank $(q+K-1)(p+K-1)+1$.
	
	Define the subgroups $B_q< \GG(S)$ and $B_p<\GG(T)$ as the full preimages of  the subgroups $A_q$ and $A_p$  correspondingly.
	Then $B_q$ has index $p+K-1$ in $\GG(S)$, and $B_p$ has index $q+K-1$ in $\GG(T)$. We claim that the groups $B_q$ and $B_p$ are isomorphic to $\GG(R)$, where $R=T((v_1,k_1),\dots, (v_l,k_l);r)$ for $r=(q+K-1)(p+K-1)+1-K$.
	
	Indeed, since $A_q$ has rank $(q+K-1)(p+K-1)+1$, one can see that $A_q$ has a free basis which includes the powers of all the $K$ pivot generators, as well as some other elements $h_1, \ldots, h_r$, for $r$ as above, and similar for $A_p$.
	
	Similar to the proof of Proposition \ref{prop:1}, it follows that both $B_q$ and $B_p$ split as fundamental groups of the star of groups with $K=k_1+\dots +k_l$ leaves, which are the induced splittings with respect to the reduced centralizer splittings of $\GG(S)$ and $\GG(T)$ respectively, and that these splittings of $B_q$ and $B_p$ are both isomorphic to the reduced centralizer splitting of $\GG(R)$.
	  Then the claim of the lemma holds with $r=(q+K-1)(p+K-1)+1-K$.
\end{proof}

\begin{theorem}\label{thm:char}
	Let $T=T((m_1,k_1), \dots, (m_l,k_l);p)$ and $S=T((n_1,r_1),\dots, (n_l,r_l);q)$ be two trees of diameter 4. Suppose that $\frac{m_i}{n_i}=\frac{m_j}{n_j}$ for all $i,j=1,\dots,l$. Then the group $\GG(T)$ is commensurable to $\GG(S)$.
\end{theorem}
\begin{proof}
	By Propositions \ref{prop:1} and \ref{prop:2} it suffices to prove the statement in the case when $l \geq 2$, $k_i=r_j=1$ for $i,j=1,\dots, l$ and $p=q=0$, and in the case when $l=1, \: k_1=r_1=2$ and $p=q=0$.
	
	Consider the first case, the proof in the second case is analogous. Let $\frac{m_i}{n_i}=\frac{m}{n}$. Consider the homomorphism $f_n:\GG(T)\to \BZ_n$ induced by the map $c_T \mapsto 1$ and $x\mapsto 0$, where $c_T$ is the center of the tree $T$ and $x$ is any other canonical generator of $\GG(T)$. Let $\GG_n$ be the kernel of $f_n$. Similarly, let $f_m:\GG(S)\to \BZ_m$ be the homomorphism induced by the map $c_S\mapsto 1$ and $y\mapsto 0$, where $c_S$ is the center of the tree $S$ and $y$ is any other canonical generator of $\GG(S)$, and let $\GG_m$ be the kernel of $f_m$.
	
	By Bass-Serre theory, it is not difficult to see that $\GG_n \simeq \GG(T((nm_1,1), \dots, (nm_l,1);0))$ and $\GG_m \simeq \GG(T((m n_1,1), \dots, (m n_l,1);0))$. Since $n m_i=m n_i$ for all $i=1, \dots, l$, it follows that $\GG_n\simeq \GG_m$ and hence $\GG(T)$ and $\GG(S)$ are commensurable.
\end{proof}

	We now turn our attention to the description of minimal elements in the commensurability classes. We first record that the RAAGs defined by paths of length $3$ and $4$ are commensurable. This fact is not new and was mentioned to us by T.~Koberda.
	
	\begin{prop} \label{prop:a1}
		$\mathbb{G}(P_3)$ is commensurable with $\mathbb{G}(P_4)$.
	\end{prop}
	\begin{proof}
		Let $\mathbb{G}(P_3)=\langle x, p_1, c,p_2\mid [x,p_1], [p_1,c],[c,p_2]\rangle$ and let $\varphi: \mathbb{G}(P_3) \to \mathbb{Z} / (2\mathbb{Z})$ be the homomorphism defined by the map
		$$
		x \to 0 \ \
		p_1 \to 0 \ \
		c \to 0 \ \
		p_2 \to 1
		$$
		Set $H = \ker \varphi$. It is clear that $H$ is an index 2 subgroup of $\mathbb{G}(P_3)$.
		
		Let $\varphi': \mathbb{G}(P_4) \to \mathbb{Z} / (2\mathbb{Z})$ be the homomorphism defined by the map
		$$
		x_1 \to 0 \ \
		p_1 \to 1 \ \
		c \to 0 \ \
		p_2 \to 1 \ \
		x_2 \to 0
		$$
		and let $H' = \ker \varphi'$, where $x_1$ and $x_2$ are the leaves of $P_4$ at the pivots $p_1$ and $p_2$ correspondingly. It is clear that $H'$ is an index 2 subgroup of $\mathbb{G}(P_4)$.
		
		Straightforward application of Reidemeister-Schreier technique shows that $H=\langle x,p_1,c,p_1^{p_2},x^{p_2}, p_2^2 \rangle$ and $H'=\langle x_1, p_1^2, c, p_2^2, x_2, p_1p_2\rangle$ are isomorphic to $\mathbb{G}(\Delta)= \langle a,b,c,d,e,f \mid [a,b]=1, [b,c]=1, [c,d]=1, [d,e]=1, [c,f]=1\rangle$.
	\end{proof}

 We deduce the following results.

 \begin{theorem}[Characterisation of commensurability classes] \label{thm:crit}
 	Let $T$ and $T'$ be two finite trees of diameter 4,  $T=T((d_1,k_1), \dots, (d_l,k_l);q)$ and $T'=T((d_1',k_1'), \dots, (d_{l'}',k_{l'}');q')$. Let $\GG=\mathbb{G}(T)$ and $\GG'=\mathbb{G}(T')$. Consider the sets $M=M(T)$, $M'=M(T')$.
		Then $\GG$ and $\GG'$ are commensurable if and only if $M$ and $M'$ are commensurable. 	
 \end{theorem}
 \begin{proof}
 	It is a consequence of Theorem \ref{th3} and Theorem \ref{thm:char}.
 \end{proof}
	
\begin{corollary}
	The groups $\GG=\mathbb{G}(P_{m_1,m_2})$ and $\GG'=\mathbb{G}(P_{n_1,n_2})$ are commensurable if and only if $\frac{m_1}{n_1}=\frac{m_2}{n_2}$.
\end{corollary}	
	
	\begin{theorem}[Minimal RAAG in the commensurability class]\label{thm:minimal}
		Let $T=T((d_1,k_1), \dots, (d_l,k_l);q)$ be a finite tree of diameter 4. Let $\mathcal C(T)$ be the commensurability class of $\GG(T)$ and let $M=M(T)$ be as above, so $|M|=l$. Then the minimal RAAG that belongs to $\mathcal C(T)$ is either the RAAG defined by the tree $T'=T((d_1',1), \dots, (d_l',1);0)$, where $M(T')$ is minimal in the commensurability class of $M$, if $|M|>1$, or the RAAG defined by the path of diameter 3, that is $\GG(P_4)$, if $|M|=1$.
	\end{theorem}
	\begin{proof}
		It is a consequence of Proposition \ref{prop:a1} and Theorem \ref{thm:crit}.
	\end{proof}

\end{document}